\documentclass[11pt,reqno]{amsart}
\usepackage[colorlinks=true,allcolors=blue,backref=page]{hyperref}
\usepackage{color}
\usepackage{amsmath, amssymb, amsthm}

\usepackage{mathrsfs}
\usepackage{mathtools}
\usepackage[noabbrev,capitalize,nameinlink]{cleveref}
\crefname{equation}{}{}
\usepackage{fullpage}
\usepackage[noadjust]{cite}
\usepackage{graphics}
\usepackage{pifont}
\usepackage{tikz}
\usepackage{tikz-cd}
\usepackage{bbm}
\usepackage[T1]{fontenc}

\usetikzlibrary{arrows.meta}

\usepackage{environ}
\usepackage{framed}
\usepackage{url}
\usepackage[linesnumbered,ruled,vlined]{algorithm2e}
\usepackage[noend]{algpseudocode}
\usepackage[labelfont=bf]{caption}
\usepackage{cite}
\usepackage{framed}
\usepackage[framemethod=tikz]{mdframed}
\usepackage{appendix}
\usepackage{graphicx}
\usepackage[textsize=tiny]{todonotes}
\usepackage{tcolorbox}
\usepackage{enumerate}
\usepackage[shortlabels]{enumitem}
\allowdisplaybreaks[1]

\apptocmd{\sloppy}{\hbadness 10000\relax}{}{} 

\crefname{algocf}{Algorithm}{Algorithms}

\crefname{equation}{}{} 
\crefname{conjecture}{Conjecture}{Conjectures} 
\AtBeginEnvironment{appendices}{\crefalias{section}{appendix}} 

\crefformat{enumi}{#2#1#3}
\crefrangeformat{enumi}{#3#1#4 to~#5#2#6}
\crefmultiformat{enumi}{#2#1#3}%
{ and~#2#1#3}{, #2#1#3}{ and~#2#1#3}

\usepackage[color,final]{showkeys} 

\colorlet{refkey}{orange!20}
\colorlet{labelkey}{blue!30}

\crefname{algocf}{Algorithm}{Algorithms}

\numberwithin{equation}{section}
\newtheorem{theorem}{Theorem}[section]
\newtheorem{proposition}[theorem]{Proposition}
\newtheorem{lemma}[theorem]{Lemma}
\newtheorem{claim}[theorem]{Claim}

\crefname{claim}{Claim}{Claims}

\newtheorem*{question*}{Question}

\theoremstyle{definition}
\newtheorem{definition}[theorem]{Definition}

\newtheorem*{definition*}{Definition}

\theoremstyle{remark}
\newtheorem*{remark}{Remark}
\newtheorem*{remarks}{Remarks}

\newcommand{\snorm}[1]{\lVert#1\rVert}

\newcommand{\legendre}[2]{\left(\frac{#1}{#2}\right)}

\newcommand{\mb}{\mathbb}

\newcommand{\mbm}{\mathbbm}
\newcommand{\mc}{\mathcal}
\newcommand{\mf}{\mathfrak}

\newcommand{\on}{\operatorname}
\newcommand{\wh}{\widehat}
\newcommand{\wt}{\widetilde}

\newcommand{\eps}{\varepsilon}

\newcommand\Supp{\operatorname{Supp}}
\newcommand\cubefull{{\textnormal{cube-full}}}

\newcommand\WW{W}

\renewcommand{\le}{\leqslant}
\renewcommand{\ge}{\geqslant}
\renewcommand{\Re}{\on{Re}}
\renewcommand{\Im}{\on{Im}}

\newcommand\vecbeta{\vec{\beta}}

\newcommand\Z{\mathbf{Z}}

\newcommand\C{\mathbf{C}}
\newcommand\R{\mathbf{R}}
\newcommand\N{\mathbf{N}}

\allowdisplaybreaks

\newcommand{\md}[1]{\ensuremath{(\operatorname{mod}\, #1)}}
\newcommand{\mdsub}[1]{\ensuremath{(\mbox{\scriptsize mod}\, #1)}}
\newcommand{\mdlem}[1]{\ensuremath{(\mbox{\textup{mod}}\, #1)}}
\newcommand{\mdsublem}[1]{\ensuremath{(\mbox{\scriptsize \textup{mod}}\, #1)}}

\title{New bounds for the Furstenberg--S\'{a}rk\"{o}zy theorem}

\author[A1]{Ben Green}
\address{Mathematical Institute, Andrew Wiles Building, Radcliffe Observatory Quarter, Woodstock Rd, Oxford OX2 6QW, UK}
\email{ben.green@maths.ox.ac.uk}

\author[A2]{Mehtaab Sawhney}
\address{Department of Mathematics, Columbia University, New York, NY 10027}
\email{m.sawhney@columbia.edu}

\begin{document}

\begin{abstract}
Suppose that $A \subset \{1,\dots, N\}$ has no two elements differing by a square. Then $|A| \ll N e^{-c\sqrt{\log N}}$. 
\end{abstract}

\maketitle

\setcounter{tocdepth}{1}
\tableofcontents
\section{Introduction}

Our main theorem is the following. Throughout the paper, we write $[X] = \{1,\dots, X\}$ for positive integer $X$.

\begin{theorem}\label{thm:main}
There exists $c_0 >0$ for which the following holds. Let $X \ge 10$ and suppose that $A\subseteq [X]$ has no solutions to $a_1-a_2 = n^2$ with $a_1,a_2\in A$ and $n\in \N$. 
\[|A|\le  X\exp(-c_0\sqrt{\log X}).\]
\end{theorem}

\subsection{History}
Denote by $s(X)$ the size of the largest subset of $[X]$ not containing two distinct elements differing by a square. Lov\'{a}sz conjectured that $s(X) = o(X)$. Independently, Furstenberg \cite[Theorem~1.2]{Fur77} and S\'ark\"ozy \cite{sarkozy-i} established the conjecture. Furstenberg's work was based on methods from ergodic theory and in particular relies on his eponymous correspondence principle (\cite{Fur77} is also where Furstenberg famously gave an alternate proof Szemer\'edi's theorem).

S\'ark\"ozy's proof was based on Fourier analysis and therefore, unlike Furstenberg's work, gave a rate on $s(X)$. In particular, S\'ark\"ozy established that 
\[s(X) \ll X (\log \log X)^{2/3}(\log X)^{-1/3}.\]
This improvement of $(\log X)^{-1/3+o(1)}$ over the trivial bound was further enhanced to $(\log X)^{-\omega(X)}$ for some function $\omega(X) \rightarrow \infty$ by Pintz, Steiger, and Szemer\'edi \cite{pintz-steiger-szemeredi}. In particular, Pintz, Steiger, and Szemer\'edi \cite{pintz-steiger-szemeredi} proved that 
\[s(X)\ll X(\log X)^{-\frac{1}{12}\log\log\log\log X}.\]
Bloom and Maynard \cite{bloom-maynard} subsequently showed that one may take 
\[s(X)\ll X(\log X)^{-c\log\log\log X}\] 
for a suitable constant $c>0$, which was the best known bound before the first version \cite{first-version} of this work, which obtained a bound of the shape 
\[ s(X) \ll X \exp(-(\log X)^{1/4 - o(1)}).\]

Regarding lower bounds, a standard probabilistic deletion argument allows one to prove that $s(X) \gg X^{1/2}$. Erd\H{o}s conjectured that $s(X)\ll X^{1/2}(\log X)^{O(1)}$ but this was disproven by S\'ark\"ozy \cite{Sar78}. However S\'ark\"ozy \cite{Sar78} still conjectured that $s(X)\ll X^{1/2+o(1)}$. This was disproven by Ruzsa \cite{ruzsa} who established that $s(X) \gg X^{1/2 + \kappa}$ with $\kappa := \frac{\log 7}{2\log 65} \approx .233077$. This was refined in work of Lewko \cite{lewko}, who showed that $\kappa = \frac{\log 12}{2\log 265}\approx .233412$ is permissible. It remains a major open question to decide whether $s(X) \ll X^{1 - c}$ for some constant $c > 0$, or not. (The authors of the present paper would speculate that the answer is yes.)

We remark that the greedy construction starting from $0$ gives a set $\{0, 2, 5, 7, 10, 12, 15, 17, 20\}$ which appears experimentally to have around $n^{0.7}$ elements up to $n$. Analysing the behaviour of this sequence rigorously seems an interesting problem. This sequence is discussed in the Online Encyclopedia of Integer Sequences \url{https://oeis.org/A030193} as sequence A030193.

We remark that a number of variants of Furstenberg--S\'ark\"ozy theorem have been proven where one replaces the set of squares by the set $\{ P(n) : n \in \Z\}$ with different polynomials $P$. The most general condition under which a variant of the Furstenberg--S\'ark\"ozy theorem holds is when $P(x)$ is intersective;\footnote{$P(x)$ is said to be intersective if for all $q\ge 1$ there is $n\in \Z$ such that $q|f(n)$.} this was proven in work of Kamae and Mend\`{e}s France \cite{KM78}. A number of results have adapted the techniques used for squares to more general univariate polynomials; these include \cite{BPPS94,HLR13,Ric19,Ara23}. It appears very likely that the techniques used to prove \cref{thm:main} adapt routinely to more general polynomials; however we do not address these questions here. 

The survey of L\^e \cite{hoang-survey} gives a comprehensive overview of related questions.

The overall approach to the proof of \cref{thm:main} is the density increment strategy, as used in all previous papers giving an explicit bound for $s(X)$. The basic idea is to establish a dichotomy, stating that if $A \subset [X]$ is a set of size $\alpha$ then either it contains two elements differing by a square, or else we may pass to a long subprogression with square common difference on which the density of $A$ is appreciably larger. Iteration of this argument (and careful control of the parameters) leads to bounds on $s(X)$.

\subsection{Arithmetic level $d$ inequality}

As with previous works, we make heavy use of the Fourier transform (more specifically, a form of the Hardy-Littlewood circle method) in order to obtain the necessary density increment. Our primary new input for \cref{thm:main} is the following result, which potentially has other applications.
If $f : \Z \rightarrow \C$ is a function, we define the Fourier transform as
\[ \wh{f}(\theta) := \sum_{x \in \Z} f(x) e(-\theta x).\] 
\begin{theorem}\label{thm:level-d-int}
Set $C_0 := 2^{13}$. Let $\alpha \in (0,1/2)$ and let $X \ge 1$ be parameters with $\alpha > 2X^{-1/2}$. Let $\mc{Q}$ be a set of pairwise coprime positive integers such that $\max_{q \in \mc{Q}} q \le X^{1/32\log(1/\alpha)}$. Let $1 \le d\le 2^{-7}\log(1/\alpha)$. Let $f : [X] \rightarrow \C$ be a function with such that $|f(x)|\le 1$ for all $x$. Then either
\begin{equation}\label{main41}\sum_{\substack{S \subseteq \mc{Q} \\ |S| = d}}\sum_{\substack{a \md{\prod_{q \in S} q}  \\ q \in S \Rightarrow q \nmid a}}\Big|\wh{f}\Big( \frac{a}{\prod_{q \in S} q} \Big)\Big|^2 \le \alpha^2 X^2 \Big(\frac{C_0\log(1/\alpha)}{d}\Big)^d,\end{equation}
or else for some set $S \subseteq \mc{Q}$, $1 \le |S| \le 2 \log(1/\alpha)$, and for some $r \in \Z$,
the average of $|f(x)|$ on the progression $P = \{ x \in [X] : x \equiv r \md{\prod_{q \in S} q}\}$ is greater than $2^{|S|} \alpha$.
\end{theorem}
\begin{remark}
At the cost of changing $C_0$, one may replace the condition of finding a progression of density $2^{|S|}\alpha$ with a progression of density $\lambda^{|S|}\alpha$ for any fixed $\lambda\ge 1$.
\end{remark}

The proof of \cref{thm:level-d-int} makes very substantial use of the results and ideas in a paper of Keller, Lifshitz and Marcus \cite{KLM23} concerning \emph{global hypercontractivity} in product spaces. This is the culmination of body of work, of which we mention in particular the earlier work of Keevash-Lifshitz-Long-Minzer \cite{kllm}. We remark in our context, a ``local'' function is one which strongly correlates with an arithmetic progression with common difference which is a product of ``few'' elements of $\mc{Q}$. A global function by contrast does not correlate strongly with such functions. We will introduce these techniques in the next section. 

We remark that when applying \cref{thm:level-d-int} to derive a density increment from arithmetic information such as a weighted average of Fourier coefficients $|\wh{1_A}(a/q)|^2$ being large, it has proven to be quantitatively advantageous to use many separate applications of \cref{thm:level-d-int} with different sets $\mc{Q}$. In particular, in our density increment argument we divide into different types of $q$ based on the structure of the prime factorization of $q$, and for each factorization type use a different $\mc{Q}$. Furthermore in our analysis it turns out to be advantageous to handle `small' prime factors of $q$ by breaking $[X]$ in various smaller subprogressions, and to employ a device we term `random sparsification' in order to effectively reduce the number of prime factors of $q$ in certain ranges.

\subsection{Structure of the paper}

The structure of the paper is as follows. In \cref{sec2}, we introduce the key concepts of global hypercontractivity and state the main result from \cite{KLM23} that we will use. In \cref{sec3} we set up the language of \emph{lifting}, which will ultimately allow us to deduce \cref{thm:level-d-int}, which concerns functions $f : [X] \rightarrow \C$, from the main results of \cite{KLM23}, which concern functions on large cyclic groups. In particular we prove a key lemma, \cref{lem32}, which will allow us to transfer the results effectively. In \cref{sec4} we prove \cref{thm:level-d-int}, relying heavily on the ideas in \cite[Section 5]{KLM23}, and also using \cite[Corollary 4.7]{KLM23} essentially as a black box. With \cref{thm:level-d-int} established, we turn to the proof of \cref{thm:main}. In \cref{sec5} we define a smooth weighting of the squares and prove some exponential sum estimates for it; this is essentially standard material. In \cref{sec6} we state the main density increment result, \cref{density-increment-calc}, and show that it implies \cref{thm:main}. This section will also be familiar to experts.

In \cref{sec7-new,sec8} we establish the density increment. In \cref{sec7-new} we show how various applications of \cref{thm:level-d-int} may be used to estimate pieces of the exponential sums that arise naturally when considering sets with no square difference. Here we give the details of the random sparsification device and the idea of handling small prime factors by passing to subprogressions, described at the end of the last subsection. Finally, in \cref{sec8} we complete the argument by piecing together the bounds from \cref{sec7-new}. This necessitates a careful analysis involving upper bounds for lattice points in certain polytopes.

To conclude the main part of the paper, in \cref{sec7} we make some self-contained comments about limitations to the density increment strategy for proving bounds on the Furstenberg-S\'ark\"ozy theorem. Finally, in \cref{appA} we remark on some slight adjustments that need to be made to the statement of \cite[Corollary 4.7]{KLM23} so that it fits our requirements.

\subsection{Notation} Most of our notation is very standard. We write $\mathbf{1}_S$ for the characteristic function of the set $S$, defined by $\mathbf{1}_S(x) = 1$ of $x \in S$ and $0$ otherwise. If $S$ is finite, and if $f : S \rightarrow \C$ is a function, we write $\mb{E}_S f$ or $\mb{E}_{x \in S} f(x)$ for $|S|^{-1} \sum_{x \in S} f(x)$.

For $x \in \R$, $\Vert x \Vert_{\R/\Z}$ denotes the distance to the nearest integer. As is standard, we let $e(\theta) = e^{2\pi i\theta}$.

We normalise the Fourier transform on finite abelian groups using the normalised counting measure on physical space and the counting measure in frequency space. That is, we define $\wh{f}(\chi) = \mb{E}_{x \in G} f(x) \overline{\chi(x)}$ for a character $\chi$ on $G$, and the inversion formula reads $f(x) = \sum_{\chi \in \wh{G}} \wh{f}(\chi) \chi(x)$. 

Similarly, if $G$ is a finite abelian group and if $f : G \rightarrow \C$ is a function then the $L^p$-norms are defined with respect to normalised counting measure, thus $\Vert f \Vert_p^p := \mb{E}_{x \in G} |f(x)|^p$ for $p \ge 1$.

If $\Sigma$ is a finite set, we denote by $B(\Sigma)$ the space of functions $f : \Sigma \rightarrow \C$.

We use standard asymptotic notation for functions on $\N$ throughout. Given functions $f=f(x)$ and $g=g(x)$, we write $f=O(g)$, $f\ll g$, or $g\gg f$ to mean that there is a constant $C$ such that $|f(x)|\le Cg(x)$ for sufficiently large $x$. Subscripts indicate dependence on parameters.

On many occasions we will see quantities such as $\exp(\log(1/\alpha)^C)$. This means $\exp\big( (\log(1/\alpha))^C\big)$ (which we believe is the natural interpretation), but to avoid notational clutter we suppress the extra pair of brackets.

\subsection{Acknowledgements}
MS thanks Dor Minzer and Dmitrii Zakharov for useful conversations.

We thank James Maynard, Sarah Peluse and Sean Prendiville for comments and corrections on the first version \cite{first-version} of the paper.

BG is supported by Simons Investigator Award 376201. This research was conducted during the period MS served as a Clay Research Fellow. 

\section{Statement of hypercontractivity results}\label{sec2}

Our main ingredient from other work will be a result on so-called global hypercontractivity from \cite{KLM23}. We state that result, in the language we will use in this paper, here. The general setup will be that we have a set $\mc{Q}$ of pairwise coprime positive integers, and we consider the product group
\[ G_{\mc{Q}} := \prod_{q \in \mc{Q}} \Z/q\Z \cong \Z/\big(\prod_{q \in \mc{Q}} q\big)\Z\] with the natural isomorphism given by the Chinese remainder theorem. The groups $\Z/q\Z$ will correspond to the sets $\Omega_i$ in \cite{KLM23}.  

We will consider several different linear operators from $B(G_{\mc{Q}})$ to itself. For many of them it seems most natural in our context to define them on the Fourier side, though in \cite{KLM23} (where the sets $\Omega_i$ need not have any group structure) the authors proceed combinatorially.

To define these operators, we fix some notation for the Fourier transform on $G_{\mc{Q}}$. We identify the dual of $\Z/q\Z$ with $\frac{1}{q}\Z /\Z \subset \R/\Z$, with $\theta \in \frac{1}{q}\Z/\Z$ corresponding to the character which, with a slight abuse of notation, we write $x \mapsto e(\theta x)$. We may then identify $\wh{G}_{\mc{Q}}$ with $\bigoplus_{q \in \mc{Q}} (\frac{1}{q}\Z/\Z) \subset \R/\Z$. If $\xi = \sum_{q \in \mc{Q}} \xi_q/q$ then we define $\Supp(\xi) := \{ q : \xi_q \neq 0 \md{q}\}$ and $|\xi| = |\Supp(\xi)|$.

All of our operators will be Fourier multiplier operators of the form $\Phi e(\xi \cdot) = w(\xi) e(\xi \cdot )$ for some $w(\xi) \in \R$. Specifically, for all $S \subseteq \mc{Q}$ we define
\begin{itemize}
    \item the \emph{averaging} operator $E_S$, with $w(\xi) = 1$ if $\xi_q = 0$ for all $q \in S$, and $w(\xi) = 0$ otherwise;
    \item the \emph{laplacian} operator $L_S$, with $w(\xi) = 1$ if $\xi_q \neq 0$ for all $q \in S$, and $w(\xi) = 0$ otherwise.
\end{itemize}
Moreover, we define
\begin{itemize}
    \item for integer $d$, the Fourier level $d$ operator $W_d$, with $w(\xi) = 1$ if $|\xi| = d$ and $w(\xi) = 0$ otherwise;
    \item for $\rho \in (0,1)$, the noise operator $T_{\rho}$, with $w(\xi) = \rho^{|\xi|}$.
\end{itemize}
\begin{remarks} These definitions are consistent with those in \cite{KLM23}. The action of these operators on general functions is easily computed using linearity and the Fourier inversion formula; for instance
\[ T_{\rho} f(x) = \sum_{\xi} \rho^{|\xi|}\wh{f}(\xi) e(\xi  x).\] Note that $W_0 f = \wh{f}(0) = \mb{E} f$ and that $W_d f = 0$ for $d < 0$.
\end{remarks}

We record some basic properties of these needed later in the paper. 

\begin{lemma}
We have
\begin{equation}\label{es-ls} L_S = \sum_{T\subseteq S}(-1)^{|T|}E_T,\end{equation}
and
\begin{equation}\label{avg-e} E_S f(x) = \mb{E}_{y \equiv x \mdsublem{\prod_{q \in \mc{Q} \setminus S} q}} f(y).\end{equation}
\end{lemma}
\begin{proof}
The statement \cref{es-ls} is immediate from the identity
\[ \prod_{q \in S} \big(1 - \mathbf{1}_{\xi_q = 0} \big) = \sum_{T \subseteq S} (-1)^T \prod_{q \in T} \mathbf{1}_{\xi_q = 0}.\]
To obtain \cref{avg-e}, we expand $f$ on the Fourier side and use that 
\[ \mb{E}_{y \equiv x \mdsublem{\prod_{q \in \mc{Q} \setminus S}q}} e(\xi y) = e(\xi  x) \] if $\xi_q = 0$ for all $q \in S$, and $0$ otherwise.
\end{proof}

To state the key global hypercontractivity result (\cref{thm:input-hypercontractive} below) we need the notions of \emph{global function} and \emph{noise operator}. The notion of a function being global (as considered in \cite{KLM23}) is what the authors of that paper call a `derivative-based' notion of globalness, which is somewhat different to the corresponding notions considered in \cite{kllm}. These notions however are closely related; see \cite[Lemma~4.9]{KLM23}. 
To define the derivative operators we first need to define specialisation operators. Let $S \subset \mc{Q}$ and let $x \in G_S$. Then we define the specialisation operator
\[ \sigma_{S \rightarrow x} : B(G_{\mc{Q}}) \rightarrow B(G_{\mc{Q} \setminus S}) \] via
\[ (\sigma_{S \rightarrow x}f)(y) = f(x, y)\] for $y \in G_{\mc{Q} \setminus S}$. (Here, of course, we have identified $G_{\mc{Q}} = G_S \times G_{\mc{Q} \setminus S}$ in the obvious way.)

Now we come to the key definition of \emph{derivatives}.

\begin{definition}
Let $S \subseteq \mc{Q}$ and $x \in G_S$. The we define the derivative operator $D_{S, x} : B(G_{\mc{Q}}) \rightarrow B(G_{\mc{Q} \setminus S})$ by $D_{S,x} = \sigma_{S \rightarrow x} L_S$. That is, $(D_{S, x} f)(y) := (L_S f)(x, y)$.
\end{definition}

To keep notation concise, we will often write $D_{S}$ for a derivative operator where the value of $x \in G_S$ can be arbitrary but whose exact nature is unimportant.

Finally we define the derivative-globalness which will be used to state the primary technical input from \cite{KLM23}. 

\begin{definition}\label{def:deriv-global}
Let $\mc{Q}$ be a set of pairwise coprime integers. Let $r \ge 0$. Let $f \in B(G_{\mc{Q}})$. We say that $f$ is $(r,\gamma)$-derivative-global if for all $S\subseteq \mc{Q}$ and $x\in G_S$ we have
\[\snorm{D_{S,x_S}f}_{2}\le r^{|S|}\gamma.\] 
\end{definition}
\begin{remarks}
We make the following comments:
\begin{enumerate}
  \item In \cite{KLM23} this notion is called $(r,\gamma)$-$L_2$-global, there also being notions of $L_p$-global for other values of $p$ (and also a notion of $(r,\gamma,d)$-$L_p$-global), and the authors choosing to not explicitly use the word `derivative' in this context. We do not need the notions with extra parameters here. 
  \item in the case $r = 0$, we should interpret $r^{0} = 1$; 
  \item If $|S| \ge 1$, any derivative $D_{S,x}$ of a constant function is zero. (This follows since $L_S$ of any constant function is zero.)
\end{enumerate}
\end{remarks}

We now state a key lemma concerning derivatives and Fourier weight operators. This is stated during the proof of \cite[Theorem~5.2]{KLM23}, but we give a proof for convenience. 

\begin{lemma}\label{eq:deriv}
Let $S \subseteq \mc{Q}$ be a set of size $r$, $x \in G_S$ and $d$ be a nonnegative integer. Then we have $D_{S,x} W_{d} = W_{d-r} D_{S,x}$.
\end{lemma}
\begin{proof}
Since all operators are linear, it suffices to check this when $f = e(\xi \cdot)$ for some $\xi$. Write $\xi' = (\xi_q)_{q \in S}$ and $\xi'' = (\xi_q)_{q \in \mc{Q} \setminus S}$. Writing $G_{\mc{Q}} = G_S \times G_{\mc{Q} \setminus S}$ as before, we have $e(\xi \cdot (x, y)) = e(\xi' x) e(\xi'' y)$. Write $\chi_{\xi} = e(\xi \cdot)$ and similarly for $\chi_{\xi'}, \chi_{\xi''}$. We have
\begin{equation}\label{der-1} D_{S,x} W_{d} \chi_{\xi}(y) = L_S W_{d}\chi_{\xi}(x,y)  = \mathbf{1}_{S \subseteq \Supp(\xi)} \mathbf{1}_{|\xi| = d} \chi_{\xi}(x,y).\end{equation} On the other hand note that 
\[ D_{S,x} \chi_{\xi}(y) = L_S \chi_{\xi}(x,y) = \mathbf{1}_{S \subseteq \Supp(\xi)}\chi_{\xi}(x,y) = \mathbf{1}_{S \subseteq \Supp(\xi)}\chi_{\xi'}(x) \chi_{\xi''}(y),\] so
\begin{equation}\label{der-2} W_{d - r} D_{S,x}\chi_{\xi}(y) = \mathbf{1}_{S \subseteq \Supp(\xi)}\chi_{\xi'}(x) \mathbf{1}_{|\xi''| = d - r} \chi_{\xi''}(y) = \mathbf{1}_{S \subseteq \Supp(\xi)} \mathbf{1}_{|\xi''| = d - r} \chi_{\xi}(x,y). \end{equation}
Comparing \cref{der-1} and \cref{der-2}, the result now follows from the observation that $\mathbf{1}_{S \subseteq \Supp(\xi)} \mathbf{1}_{|\xi| = d} = \mathbf{1}_{S \subseteq \Supp(\xi)} \mathbf{1}_{|\xi''| = d - r}$.
\end{proof}

The following is the main input from \cite{KLM23} that we will need. It is (essentially, as discussed in the remarks below, and in more detail in \cref{appA}) Corollary~4.7 (1.) in that paper. We remark that if one is aiming for only some quasi-polynomial bound in \cref{thm:main}, then \cite[Theorem~7.10]{kllm} could be used in place of this result. Note, however, that we also use ideas from Section 5 of \cite{KLM23}.

\begin{theorem}\label{thm:input-hypercontractive}
Let $r,\gamma>0$, and $p\ge 2$. Suppose that $f \in B(G_{\mc{Q}})$ is an $(r,\gamma)$-derivative-global function. Suppose that
\begin{equation}\label{rho-cond-23} \rho \le \frac{\min(r^{-(p-2)/p}p^{-1}, p^{-1/2})}{3\sqrt{2}}.\end{equation}
Then we have 
\[\snorm{T_{\rho} f}_p^p\le \snorm{f}_2^2\gamma^{p-2}.\]
\end{theorem}
\begin{remarks}
\cref{thm:input-hypercontractive} is a version of \cite[Corollary~4.7]{KLM23} with the product space $\Omega^n$ replaced by $\Omega_1 \times \cdots \times \Omega_n$, where the $\Omega_i$ are the factors $\Z/q\Z$ of $G_{\mc{Q}}$. Though the results of \cite{KLM23} are only stated in the case $\Omega_1 = \cdots = \Omega_n = \Omega$, all the definitions and proofs adapt routinely to the case of distinct factors. Also, \cite[Corollary~4.7]{KLM23} is stated for functions taking values in $\R$, whereas we need it for functions taking values in $\C$. Again, the proof can be straightforwardly modified. For completeness, in \cref{appA} we give a deduction of the statements we require from the actual statement of \cite[Corollary~4.7]{KLM23}, using simple tensor power tricks.

In our application, we choose parameters so that $p$ is an even integer and to balance the two terms in the condition on $\rho$. Specifically, one may check that if one takes $m := \lceil r^{-2}\rceil$, $p := 2m$ and $\rho := \frac{1}{20} m^{-1/2}$ then the condition \cref{rho-cond-23} is satisfied for $m \ge 4$.
\end{remarks}

\section{Liftings}\label{sec3}

The hypercontractive results stated in the previous section take place in the context of finite (and possibly large) cyclic groups $G_{\mc{Q}}$. To use them to prove results about the integers, we need to do some lifting. 

\begin{definition}\label{lifting-def}
Let $X \ge 1$ be a parameter. Let $\mc{Q}$ be a set of pairwise coprime positive integers. Denote by $\pi_{\mc{Q}} : \Z \rightarrow G_{\mc{Q}}$ the natural projection map. Suppose that $P \subseteq [X]$ is an arithmetic progression with common difference coprime to all $q \in \mc{Q}$. Then we define the lifting map $\Psi_{P,\mc{Q}} : B(P) \rightarrow B(G_{\mc{Q}})$ by
\[ (\Psi_{P,\mc{Q}} f)((t_q)_{q \in \mc{Q}}) =  |P|^{-1}|G_{\mc{Q}}| \sum_{\substack{x \in P \\ x \equiv t_q \mdsub{q} \\ q \in \mc{Q}}} f(x).\]
\end{definition}
\begin{remark}
The normalization is chosen so that 
\begin{equation}\label{normalization} \mb{E}_{G_{\mc{Q}}} (\Psi_{P,\mc{Q}} f) = \mb{E}_{P} f.\end{equation}
\end{remark}

The following lemma details the relation between restrictions and lifts.

\begin{lemma}\label{cd-lem}
Let $X \ge 1$ be a parameter. Let $P \subseteq [X]$ be an arithmetic progression of length $\ge X^{3/4}$. Let $\mc{Q}$ be a set of pairwise coprime positive integers. Suppose that the common difference of $P$ is coprime to all $q \in \mc{Q}$. Let $S \subseteq \mc{Q}$ be such that $(\max_{q \in \mc{Q}} q)^{|S|} \le X^{1/4}$, and suppose that $T \subseteq S$. Let $a \in G_S$. Denote $\mc{Q}' := \mc{Q} \setminus S$, and let $P' := \{ x \in P : x \equiv a \md{\prod_{q \in S \setminus T} q}\}$. Then there is $\eta$, $|\eta| \le X^{-1/2}$, such that all the maps in the following diagram
\[
\begin{tikzcd}
B(P) \arrow[dd, "(1 + \eta)\iota_*", shift right] \arrow[r] \arrow[r, "\Psi_{P,\mc{Q}}", no head] & B(G_{\mc{Q}}) \arrow[d, "E_T"] \\ & B(G_{\mc{Q}}) \arrow[d, "\sigma_{S \rightarrow a}"] \\
B(P') \arrow[r, "\Psi_{P', \mc{Q}'}"] & B(G_{\mc{Q}'})           
\end{tikzcd}
\]
are well-defined, and the diagram commutes. Here $\iota_*$ is the restriction map induced by the inclusion $\iota : P' \hookrightarrow P$.
\end{lemma}
\begin{proof} 
We first verify that all maps are well-defined. The only nontrivial point here is to check that $\Psi_{P', \mc{Q}'}$ is well-defined. If the common difference of $P$ is $d$  (coprime to all $q \in \mc{Q}$) then the common difference of $P'$ is $d \prod_{q \in S \setminus T} q$, which is coprime to all $q \in \mc{Q}' = \mc{Q} \setminus S$.

Now we check the diagram commutes. All maps are linear so it suffices to check that the diagram commutes for functions $f \in B(P)$ of the form $f = \delta_b$, that is to say $f(b) = 1$ and $f(x) = 0$ for $x \neq b$. We first note that $\Psi_{P,\mc{Q}} f = |P|^{-1}|G_{\mc{Q}}|\delta_{\pi_{\mc{Q}}(b)}$. From this and \cref{avg-e} we compute that 
\[ E_T \Psi_{P, \mc{Q}} f = \frac{|G_{\mc{Q}}|}{|P||G_T|} \sum_{\substack{y \in G_{\mc{Q}} \\ y \equiv b \mdsub{\prod_{q \in \mc{Q} \setminus T} q} }} \delta_y.\]
Now suppose that $x \in G_{\mc{Q}'}$; then from the preceding we have
\[ \sigma_{S \rightarrow a} E_T \Psi_{P, \mc{Q}} (x) = E_T \Psi_{P, \mc{Q}} (x, a) = \frac{|G_{\mc{Q}}|}{|P||G_T|} \sum_{\substack{(x,a) \in G_{\mc{Q}} = G_{\mc{Q}'} \times G_S \\ (x,a) \equiv b \mdsub{\prod_{q \in \mc{Q} \setminus T} q} }} \delta_{(x,a)}. \]  The congruence condition decouples to the two conditions 
\[ x \equiv b \md{\prod_{q \in \mc{Q}'} q} \quad \mbox{and} \quad b \equiv a \md{\prod_{q \in S \setminus T} q}.\] Thus we conclude that 
\[ \sigma_{S \rightarrow a} E_T \Psi_{P, \mc{Q}}f = \frac{|G_{\mc{Q}}|}{|P||G_T|} \delta_{\pi_{\mc{Q}'} (b)}\] if $b \equiv a \md{\prod_{q \in S \setminus T} q}$, and $0$ otherwise. Traversing the diagram the other way, we see that 
\[ \Psi_{P', \mc{Q}'} \iota_* f = \frac{|G_{\mc{Q}'}|}{|P'|} \delta_{\pi_{\mc{Q}'}(b)}\] if $b \equiv a \md{\prod_{q \in S \setminus T} q}$, and $0$ otherwise.

Therefore the result holds if we define $\eta$ by
\[ (1 + \eta) \frac{|G_{\mc{Q}'}|}{|P'|} = \frac{|G_{\mc{Q}}|}{|P||G_T|}.\]
To complete the proof note that $\big| |P'| - |G_{S \setminus T}|^{-1}|P|\big| \le 1$ (here we use the fact that all $q \in S \setminus T$ are coprime to the common difference of $P$), so 
\[ |\eta| \le \frac{|G_{\mc{Q}}|}{|P||G_T||G_{\mc{Q}'}|} \le \frac{|G_S|}{|P|} \le \frac{(\max_{q \in \mc{Q}} q)^{|S|}}{|P|} \le X^{-1/2} .\]

\end{proof}

We record a lemma about Fourier coefficients and lifts (in the case $P = [X])$.

\begin{lemma}\label{fourier-lifts}
Let $X \ge 1$ be a parameter, and let $\mc{Q}$ be a set of pairwise coprime positive integers. Let $P = [X]$, let $f \in B(P)$, and let $\Psi_{P, \mc{Q}} f$ be the lift of $f$ to $B(G_{\mc{Q}})$. Let $\xi \in \wh{G}_{\mc{Q}}$, identified with $\bigoplus_{q \in \mc{Q}} (\frac{1}{q} \Z/\Z)$ as explained in the previous section. Then 
\[ \wh{\Psi_{P,\mc{Q}} f}(\xi) = \mb{E}_{x \in [X]} f(x) e(-\xi x).\]
\end{lemma}
\begin{proof} This is little more than chasing through the definitions.\end{proof}

A crucial ingredient in our arguments will be the following.

\begin{lemma}\label{lem32}
Suppose that $\mc{Q}$ is a set of pairwise coprime integers. Let $P \subseteq [X]$ be an arithmetic progression with length at least $X^{3/4}$ and common difference coprime to all $q \in \mc{Q}$. Let $d, m \ge 1$ be integers. Write $\mu$ for the pushforward of the normalised counting measure on $P$ under the natural projection $\pi_{\mc{Q}} : \Z \rightarrow G_{\mc{Q}}$: that is, $\mu$ is the measure on $G_{\mc{Q}}$ defined by $\mu(t)$ is $|P|^{-1}$ times the number of $x \in P$ with $x \equiv t \md{\prod_{q \in \mc{Q}} q}$. Let $f \in B(P)$ be a function with $\mb{E}_{P} |f| = \alpha$, and suppose that $(\max_{q \in \mc{Q}} q)^{dm} \le X^{1/16}$. Write $g := \Psi_{P,\mc{Q}} f \in B(G_{\mc{Q}})$ for the lift as defined in \cref{lifting-def}, and write $W_d$ for the Fourier level $d$ operator on $B(G_{\mc{Q}})$ as described at the beginning of \cref{sec2}. Then
\[  \int |W_{d} g|^{2m} d\mu  \le \mb{E}_{ G_{\mc{Q}}} |W_{d} g|^{2m} + \alpha^{2m} X^{-1/4}.\]
\end{lemma}
\begin{proof} Write $q_* := \max_{q \in \mc{Q}} q$ for notational convenience. Since $\mb{E}_{P} |f| = \alpha$, it follows from \cref{normalization} that $\mb{E}_{G_{\mc{Q}}} |g| = \alpha$, and thus $|\wh{g}(\xi)| \le \alpha$ for all $\xi \in \wh{G}_{\mc{Q}}$. By definition, 
\[ W_{d} g(x) = \sum_{\xi : |\xi| = d} \widehat{g}(\xi) e (\xi  x).\]
Denoting by $\mc{C}$ the complex conjugation map, we compute
\begin{align}\nonumber
\mb{E}_{G_{\mc{Q}}} (\mu - |G_{\mc{Q}}|^{-1}) |W_{d} g|^{2m}  & =  \sum_{\substack{\sum_{j=1}^{2m} (-1)^j \xi_j \neq 0 \\ |\xi_j| = d}} \Big(\prod_{j=1}^{2m} \mathcal{C}^{j}\wh{g}(\xi_j)\Big) \widehat{\mu}\Big({-\sum_{j = 1}^{2m}(-1)^j \xi_j}\Big) \\ &  \le \alpha^{2m}\sum_{\substack{\sum_{j=1}^{2m}(-1)^j\xi_j\neq 0 \\ |\xi_j| = d}} \Big|\widehat{\mu}\Big({-\sum_{j = 1}^{2m} (-1)^j\xi_j}\Big)\Big|.\label{comparison-fourier}\end{align}
To estimate the Fourier coefficients $\wh{\mu}$ we unwind the identification of $\wh{G}_{\mc{Q}}$ with $\bigoplus_{q \in \mc{Q}} (\frac{1}{q}\Z/\Z)$, writing $\xi_j = \sum_{q \in \mc{Q}} \theta_{q,j}/q$ for $j = 1,\dots, 2m$ and for integers $\theta_{q,j}$. Thus
\begin{equation}\label{fourier-gamma}
 \wh{\mu}\Big({-\sum_{j = 1}^{2m} (-1)^j\xi_j}\Big) = |G_{\mc{Q}}|^{-1} \mb{E}_{x \in P} e (\theta x),
\end{equation}
where $\theta := \sum_{q \in \mc{Q}} \sum_{j = 1}^{2m} (-1)^j\theta_{q,j}/q$.
Now observe that $\theta$ is not zero in $\R/\Z$ (since $\sum_{j=1}^{2m} (-1)^j \xi_j \neq 0)$. Also, since $|\xi_j| \le d$ for all $j$, at most $2dm$ of the $\theta_{j,q}$ are non-zero. Thus $\theta$ is a non-zero element of $\R/\Z$ with a denominator dividing some product $\prod_{i = 1}^{2dm} q_i$, $q_i \in \mc{Q}$. Let $t$ be the common difference of $P$. Since $t$ is coprime to all elements of $\mc{Q}$, $\theta t$ is also a non-zero element of $\R/\Z$ with a denominator dividing some product $\prod_{i = 1}^{2dm} q_i$, $q_i \in \mc{Q}$. Since $\prod_{i = 1}^{2dm} q_i \le q_*^{2dm} \le X^{1/4}$, it follows that $\Vert \theta t \Vert_{\R/\Z} \ge X^{-1/4}$. By summing a geometric progression we have 
\[ \big| \mb{E}_{x \in P} e(\theta x)| \le |P|^{-1} \Vert \theta t \Vert_{\R/\Z}^{-1} \le X^{-1/2}\] (here using that the length of $P$ is at least $X^{3/4}$). From this and \cref{fourier-gamma} it follows that
\[ \Big| \wh{\mu}\Big({-\sum_{j = 1}^{2m}(-1)^j  \xi_j}\Big)\Big| \le |G_{\mc{Q}}|X^{-1/2}.\]
The number of $\xi \in \wh{G}_{\mc{Q}}$ with $|\xi| = d$ is bounded above by $\binom{|\mc{Q}|}{d} q_*^d \le q_*^{2d}$, so the number of terms in the sum on the RHS of \cref{comparison-fourier} is $\le q_*^{4dm} \le X^{1/4}$, by assumption. Thus the RHS of \cref{comparison-fourier} is $\le |G_{\mc{Q}}|^{-1}\alpha^{2m}X^{-1/4}$, and the result follows upon multiplying through by $|G_{\mc{Q}}|$.
\end{proof}

\section{Proof of the arithmetic level \texorpdfstring{$d$}{} inequality}\label{sec4}

In this section, we prove \cref{thm:level-d-int}. The argument is a certain ``relativized'' version of the inductive argument present in \cite[Lemma~5.1,~Theorem~5.2]{KLM23}. The main new idea is to alter the use of H\"older's inequality in \cite[Lemma~5.1]{KLM23} (which fails to account for the sparsity of $[X]$ when embedded as a subset of $G_{\mc{Q}}$) via an appeal to \cref{lem32}.

In the proof of \cref{thm:level-d-int} we will use the following definition, which essentially encodes the failure of the second clause in \cref{thm:level-d-int}.

\begin{definition}\label{def:global}
Let $P \subset \Z$ be an arithmetic progression and let $\mc{Q}$ be a set of pairwise coprime integers, all coprime to the common difference of $P$. Let $f :P \to \C$. We say that $f$ is $(r,\alpha,d)$-integer-global with respect to $P, \mc{Q}$ if, for every subprogression $P' = \{ x \in P : x \equiv a \mdlem{\prod_{q \in S} q}\}$, where $S \subseteq \mc{Q}$ is a set of cardinality at most $d$, the average of $|f(x)|$ on $P'$ is at most $r^{|S|} \alpha$. 
\end{definition}
We note that this definition enjoys a hereditary property: if $f$ is $(r,\alpha, d)$-integer-global with respect to $P, \mc{Q}$, then it is $(r, r^{|S|} \alpha, d - |S|)$-integer-global with respect to $P', \mc{Q}'$, for any subprogression $P' \subseteq P$ cut out by a congruence condition as above and with $\mc{Q}' = \mc{Q} \setminus S$.

We now write down a claim which is somewhat stronger than \cref{thm:level-d-int}, so as to enable an inductive argument. One difference is that it involves a progression $P \subseteq [X]$, whilst \cref{thm:level-d-int} corresponds to the case $P = [X]$. This is necessary because the proof proceeds inductively by passing to subprogressions. A more important difference is that the claim encodes (via the notion of globalness, \cref{def:deriv-global}) certain bounds on derivatives which do not appear in \cref{thm:level-d-int}. Again, this is to allow an inductive argument to work.

Define for non-negative integer $d$ and real $\alpha \in (0,1)$ the quantities 
\begin{equation}\label{r-gam-defs} r_d(\alpha) := \Big(\frac{d}{\log(1/\alpha)}\Big)^{1/2}\text{ and }\gamma_d(\alpha):= \Big(\frac{C_0\log(1/\alpha)}{d}\Big)^{d/2} \alpha.\end{equation}
where $C_0 = 2^{13}$ and with the convention that $\gamma_0(\alpha) = \alpha$.

\begin{claim}\label{key-inductive-claim}
Let $X \ge 1$ and $\alpha \in (0,1]$ be parameters with $\alpha > 4  X^{-1/2}$. Write $L := \log(1/\alpha)$. Let $d \le 2^{-7} L$ be a nonnegative integer. Let $P \subseteq [X]$ be an arithmetic progression of length $\ge X^{3/4 + d/16L}$. Let $\mc{Q}$ be a set of pairwise coprime positive integers with $\max_{q \in \mc{Q}} q \le X^{1/32L}$, and with all $q \in \mc{Q}$ coprime to the common difference of $P$. Suppose that $f \in B(P)$ is $(2,\alpha, 2 L)$-integer-global. Then the function $W_{d} \Psi_{P, \mc{Q}} f \in B(G_{\mc{Q}})$ is $(r_d(\alpha), \gamma_d(\alpha))$-derivative-global in the sense of \cref{def:deriv-global}. 
\end{claim}

Before proving the claim, we deduce \cref{thm:level-d-int} from it. 

\begin{proof}[Proof of \cref{thm:level-d-int}] We seek to apply the claim with $P = [X]$; thus the conditions on the length and common difference of $P$ are then trivially satisfied. If $f$ is not $(2, \alpha, 2 \log(1/\alpha))$-integer-global then the second clause of \cref{thm:level-d-int} holds, so we may suppose $f$ does satisfy this condition. The hypotheses of the claim are therefore satisfied, and we conclude that $W_d \Psi_{P,Q} f$ is $(r_d(\alpha), \gamma_d(\alpha))$-derivative-global (see \cref{def:deriv-global}). In particular, taking $S = \emptyset$ in the definition, it follows that 
\[ \Vert W_d \Psi_{P,Q} f \Vert_2^2 \le \gamma_d(\alpha)^2  = \Big( \frac{C_0 \log(1/\alpha)}{d} \Big)^d \alpha^2.  \]
This statement is equivalent to \cref{main41}: this follows from \cref{fourier-lifts} and the definition (\cref{sec2}) of $W_d$, noting that the weight $d$ elements $\xi \in \wh{G}_Q$ are precisely the elements $\frac{a}{\prod_{q \in S} q}$, where $|S| = d$ and $q \in S \Rightarrow q \nmid a$, appearing in \cref{main41}.

\end{proof}

We now turn to the proof of \cref{key-inductive-claim} itself.

\begin{proof}[{Proof of \cref{key-inductive-claim}}]
Throughout the proof, write $L = \log(1/\alpha)$ and $\tilde f := \Psi_{P, \mc{Q}} f$ for the lift of $f$ to $B(G_Q)$. We will proceed by induction on $d$. 

The case $d = 0$ is straightforward after parsing the definitions, as we will now show. In this case, $W_{d} \wt{f} = \mb{E} \wt{f}$ is constant. From the fact that $f$ is $(2,\alpha, 2 L)$-integer-global (\cref{def:global}, applied with $S = \emptyset$) and from \cref{normalization} applied to $|f|$ we see that 
\begin{equation}\label{ef}  |\mb{E}_{G_{\mc{Q}}} \wt{f}| \le \mb{E}_{G_{\mc{Q}}} |\wt{f}| = \mb{E}_{[X]} |f|  \le \alpha.\end{equation} We are required to prove that $W_0 \wt{f}$ is $(0, \alpha)$-derivative-global, where this definition is given in \cref{def:deriv-global}; thus we must check that $\Vert D_S W_0 \wt{f} \Vert_2 \le 0^{|S|} \alpha$ for all $S$, where $0^0$ is interpreted as $1$. In the case $S = \emptyset$, this follows from the preceding discussion. In the case $|S| \ge 1$, the derivative $D_S W_0 \wt{f}$ is zero (see Remark (4) after \cref{def:deriv-global}). This completes the discussion of the case $d = 0$.

Now suppose that $d \ge 1$, and we have verified the claim for all smaller values of $d$.
Recalling \cref{def:deriv-global}, what we wish to show is that 
\begin{equation}\label{wish-to-show} \Vert D_S W_d \wt{f} \Vert_2 \le r_d(\alpha)^{|S|} \gamma_d(\alpha)\end{equation} for all $S \subseteq \mc{Q}$.
Set $k := |S|$, and write $\mc{Q}' := \mc{Q} \setminus S$. We first handle the case when $k\neq 0$. 

The crucial fact here will be \cref{eq:deriv}, which states that 
\begin{equation}\label{swap-der-weight}  D_{S} W_{d} \wt{f} = W_{d-k} D_{S} \wt{f}.\end{equation} 
If $k > d$ then both sides are zero and \cref{wish-to-show} is immediate. Therefore suppose that $k \le d$, and in the first instance suppose that $k \ge 1$.
Equation \cref{swap-der-weight} will allow us to access the inductive hypothesis. To carry this out, suppose that the derivative $D_S \wt{f}$ is $D_{S \rightarrow a} \wt{f}$ for some $a \in G_S$. Using \cref{es-ls}, we have
\[ D_S \wt{f} = \sigma_{S \rightarrow a} L_S \wt{f} = \sum_{T \subseteq S} (-1)^{|T|} \sigma_{S \rightarrow a} E_T \wt{f}.\]
We apply \cref{cd-lem}; this is valid since $(\max_{q \in \mc{Q}} q)^{k} \le X^{d/32L} \le X^{1/4}$. From the commutative diagram in that lemma we obtain
\[\sigma_{S \rightarrow a} E_T \wt{f} = (1 + \eta_T)\Psi_{P_T,\mc{Q}'} f, \] where $\eta_T = \eta_{S, T, a, P}$ satisfies $|\eta_T| \le X^{-1/2}$, each $P_T$ is a subprogression of $P$ cut out by a congruence condition modulo $\prod_{q \in S\setminus T} q$, and by slight abuse of notation we have identified $f$ with its restriction to $P_T$. Thus
\[ D_S \wt{f} = \sum_{T \subseteq S} (-1)^{|T|}(1 + \eta_T)\Psi_{P_T,\mc{Q}'} f, \] and so from \cref{swap-der-weight} (and crudely bounding $|\eta_T| \le 1$) we have
\begin{equation}\label{ind-kd}  \Vert D_S W_d \wt{f}\Vert_2 \le 2\sum_{T \subseteq S} \Vert W_{d-k} \Psi_{P_T, \mc{Q}'} f \Vert_2.\end{equation}

By the hereditary nature of the definition of integer-global (that is, the comments after \cref{def:global}) we see that $f$ is $(2, 2^k \alpha, 2 L  - k)$-integer-global with respect to $P_T, \mc{Q}'$ and therefore (since $2 \log 2 > 1$) is $(2, \alpha', 2 L')$-integer-global with respect to $P_T, \mc{Q}'$ where $\alpha' := 2^k \alpha$ and $L' := \log(1/\alpha')$.
Note that $\alpha' = 2^k \alpha \le 2^d \alpha \le 1$ by the assumption on $d$. Also, 
\[ |P_T| \ge |P| (\max_{q \in \mc{Q}} q)^{-k} -1 \ge X^{\frac{3}{4} + \frac{d - k/2}{16 L}} - 1 \ge X^{\frac{3}{4} + \frac{d - 2k/3}{16 L}}.\] (For the second bound, the crude inequality $\frac{k}{96L} \ge \frac{1}{96\log X}$ amply suffices.)
Now one may check that $\frac{d - 2k/3}{L} \ge \frac{d - k}{L'}$ (the key inequality here is that $d \le \frac{L}{3 \log 2}$, which is certainly a consequence of the assumption on $d$). Therefore
\begin{equation}\label{ind-1} |P_T|  \ge X^{\frac{3}{4} + \frac{d-k}{16 L'}}. \end{equation} 
From the inductive hypothesis (which we may apply since $k \ge 1$, noting the condition \cref{ind-1}), $W_{d-k} \Psi_{P_T, \mc{Q}'} f$ is derivative-global with parameters $(r_{d-k}(\alpha'), \gamma_{d-k}(\alpha'))$, and so from the definition in \cref{def:deriv-global} of derivative-global and \cref{ind-kd} we obtain (crudely bounding $2^{k+1} \le 4^k$)
\[ \Vert D_S W_d \wt{f} \Vert_2  \le 4^k\gamma_{d-k}(\alpha'). \]
To complete the proof of \cref{wish-to-show} (in the case $k \ge 1$) it therefore suffices to establish
\[4^{k}\gamma_{d-k}(\alpha')\le r_{d}(\alpha)^{k}\gamma_{d}(\alpha).\]
Recalling the definitions \cref{r-gam-defs} and the fact that $\alpha' = 2^k \alpha$ and expanding, we see that this is equivalent to
\begin{equation}\label{equiv-calcul} 8^{k} \Big(\frac{d}{d-k}\Big)^{(d-k)/2}\le C_0^{k/2} \Big(\frac{L}{L'}\Big)^{(d-k)/2}.\end{equation}
Since $\frac{d}{d - k} \le e^{k/(d-k)}$, we see that the LHS is at most $(8e^{1/2})^k$. Since $L \ge L'$, 
 the RHS is at least $C_0^{k/2}$. Therefore \cref{equiv-calcul} indeed holds since $C_0$ is much bigger than $64e$.

This completes the proof of \cref{wish-to-show} in the case $k \ge 1$, that is to say when $S \neq \emptyset$. For the remaining case $S = \emptyset$, the required statement \cref{wish-to-show} is
\begin{equation}\label{kprime-zero} \Vert W_d \tilde f \Vert_2 \le \gamma_d(\alpha).\end{equation}
Let $m := \lceil L/d\rceil \ge 4$. 

As in \cref{lem32}, denote by $\mu$ the pushforward of the uniform counting measure on $P$ under the natural map $\pi_{\mc{Q}} : \Z \rightarrow G_{\mc{Q}}$. We observe that the map $(\pi_{\mc{Q}})|_P$ has fairly uniform fibres: if $|G_{\mc{Q}}| \ge |P|$ then $\mu(\gamma) = |P|^{-1}$ for all $\gamma \in \pi_{\mc{Q}}(P)$, whilst if $|G_{\mc{Q}}| < |P|$ then $\pi_{\mc{Q}}$ is surjective and $\mu(\gamma)$ is either $|P|^{-1}\lfloor |P| /|G_{\mc{Q}}| \rfloor$ or $|P|^{-1}\lceil |P| /|G_{\mc{Q}}| \rceil$, hence varies by a factor of at most $2$.

Write $\Gamma := \pi_{\mc{Q}}(P)$. Noting that $\wt{f}$ is supported on $\Gamma$, we have (with inner products on $G_{\mc{Q}}$) 
\begin{equation}\label{first-wd}
\snorm{W_d \wt{f}}_2^2 = \langle \wt{f}, W_d\wt{f}\rangle  =\langle \wt{f} , \mathbf{1}_{\Gamma} W_d \wt{f} \rangle  \le \Vert \wt{f} \Vert_{2m/(2m - 1)} \Vert \mathbf{1}_{\Gamma} W_d \wt{f} \Vert_{2m}.
\end{equation}
Now since $|f| \le 1$, we have $|\tilde f| \le \Vert \mu \Vert_{\infty} |G_{\mc{Q}}|$ by the definition of the lift $\tilde f = \Psi_{P,\mc{Q}} f$ (\cref{lifting-def}) and of the measure $\mu$. We also have $1_{\Gamma} \le 2 \Vert \mu \Vert_{\infty}^{-1} \mu$ by the above remarks about the uniform fibres of $\pi_{\mc{Q}}$. Substituting in to \cref{first-wd} and using the fact (see \cref{ef}) that $\mb{E} |\tilde f| \le \alpha$, we therefore obtain
\begin{equation}
\snorm{W_d \wt{f}}_2^2 \le 2^{1/2m}\alpha^{(2m-1)/2m}  \Big(\int |W_d \wt{f}|^{2m} d\mu \Big)^{1/2m}. \label{key-compare}
\end{equation}
Using $d \lceil L/d\rceil \le 2L$ (which holds since $d \le 2^{-7} L$) and the assumption $\max_{q \in \mc{Q}} q \le X^{1/32 L}$, we see that the condition $(\max_{q \in \mc{Q}} q)^{dm} \le X^{1/16}$ required in \cref{lem32} is satisfied. With the notation in use here, that lemma states that
\[ \int |W_d \wt{f}|^{2m} d\mu  \le \mb{E}_{G_{\mc{Q}}} |W_d \tilde f|^{2m} + \alpha^{2m} X^{-1/4}. \]
 which implies that
\[  \Big(\int |W_d \wt{f}|^{2m} d\mu\Big)^{1/2m} \le 2^{1/2m} \max(\snorm{W_d \wt{f}}_{2m}, X^{-1/4m}\alpha  ). \]
Thus, substituting in to \cref{key-compare} we see that 
\begin{equation}\label{two-terms} \snorm{W_d \wt{f}}_2^2\le 2^{1/m} \alpha^{(2m-1)/2m} \max(\snorm{W_d \wt{f}}_{2m}, X^{-1/4m} \alpha).\end{equation}
If the second term on the RHS of \cref{two-terms} is the larger then (since $\alpha \ge 4 X^{-1/2}$)
\[\snorm{W_d \wt{f}}_2^2\le \alpha^{2} (4\alpha^{-1}X^{-1/2})^{1/2m}\le \alpha^2\]
which certainly implies \cref{kprime-zero} (since $\gamma_d(\alpha) \ge \alpha$ in the range $d \le 2^{-7}\log(1/\alpha)$). 

Otherwise, the first term on the RHS of \cref{two-terms} is larger, that is to say
\begin{equation}\label{first-term} \snorm{W_d \wt{f}}_2^2\le 2^{1/m} \alpha^{(2m-1)/2m} \snorm{W_d \wt{f}}_{2m}.\end{equation}

Suppose from now on, as a hypothesis for contradiction, that \cref{kprime-zero} fails to hold. Then $W_d \wt{f}$ is $(r_d(\alpha),\snorm{W_d \wt{f}}_2)$ derivative-global. Indeed, we have already shown the stronger condition 
\[ \Vert D_S W_d \wt{f} \Vert_2 \le r_d(\alpha)^{|S|} \gamma_d(\alpha)\] when $|S| > 0$, and for $S = \emptyset$ the condition 
\[ \Vert D_S W_d \wt{f} \Vert_2 \le r_d(\alpha)^{|S|}\Vert W_d \wt{f} \Vert_2\] is trivial.

With an application of \cref{thm:input-hypercontractive} in mind, we choose parameters as in the comments after the statement of that result. Thus, set
\[ r := r_d(\alpha), \quad m= \lceil r^{-2} \rceil = \lceil L/d\rceil, \quad p := 2m, \quad \rho := \frac{1}{20} m^{-1/2}.\]
Set $\gamma = \Vert W_d \wt{f} \Vert_2$ so that, as discussed above, $W_d \wt{f}$ is $(r, \gamma)$-derivative-global.

Apply \cref{thm:input-hypercontractive}; the conclusion is that 
\begin{equation}\label{th23-conclu}\snorm{W_d\wt{f}}_{2m} = \rho^{-d}\snorm{T_{\rho}W_d \wt{f}}_{2m}\le \rho^{-d}\snorm{W_d \wt{f}}_2.\end{equation}

Thus \cref{first-term}, \cref{th23-conclu} give

\begin{equation}\label{wd-upper} \snorm{W_d \wt{f}}_2^2\le 2^{2/m} \alpha^{(2m-1)/m} \rho^{-2d}  \le 4(400)^d \alpha^{2 - 1/m} m^d \le 4(400 e)^d \alpha^{2} m^d .\end{equation}
On the other hand, from the assumption that \cref{kprime-zero} fails and the definition \cref{r-gam-defs} of $\gamma_d(\alpha)$ we have
\begin{equation}\label{wd-lower} \Vert W_d \wt{f} \Vert_2^2 \ge \gamma_d(\alpha)^2 \ge (C_0 (m-1))^d \alpha^2.\end{equation}
Since $C_0$ is so large, the bounds \cref{wd-upper} and \cref{wd-lower} contradict one another.

Thus, we were wrong to suppose that \cref{kprime-zero} does not hold, and this completes the proof of all cases of \cref{key-inductive-claim}.
\end{proof}

\section{Weight functions and exponential sums}\label{sec5}

In this section we define a smooth weight function on the squares and prove some exponential sum estimates for it. The estimates are slightly bespoke to the application at hand, but of essentially standard type.

\subsection{A smooth weight on the squares}

To control Fourier transforms, we will need a smooth function supported on squares. To construct this, we use a smooth bump function $w$ with the following properties.

\begin{lemma}\label{lem:bump}
There exists a smooth function $w:\R \to [0,1]$, supported on $[-1,1]$, such that $w(x) \ge \frac{1}{2}$ for $x \in [-\frac{1}{2}, \frac{1}{2}]$ and satisfying the Fourier estimate $|\wh{w}(t)|\ll e^{-\sqrt{|t|}}$ for all $t \in \R$.
\end{lemma}
\begin{proof}
The existence of such $w$ is standard and we omit the proof. One may proceed either using explicit functions as in \cite[Lemma 9]{BFI}, or via infinite convolutions of normalized characteristic functions of intervals (one being $[-\frac{3}{4}, \frac{3}{4}]$ and the rest being $[-n^{-3/2}, n^{-3/2}]$ for $n = 101,102,\dots$, say) as described in \cite[Chapter V, Lemma 2.7]{katznelson} in connection with the Denjoy-Carleman Theorem.
\end{proof}

We will also need Weyl's inequality, which we give in in the following form.

\begin{lemma}\label{weyl-consequence}
Let $X$ be sufficiently large. Let $\theta \in \R/\Z$ and suppose that $|\mb{E}_{x \in [X]} e(\theta x^2) | \ge \delta X$, where $\delta \ge X^{-1/3}$. Then there is some $q \ll (\log X)^2 \delta^{-2}$ such that $\Vert \theta q\Vert_{\R/\Z} \ll (\log X)^2 \delta^{-2}X^{-2}$.
\end{lemma}
\begin{proof}
We recall the usual statement of Weyl's inequality (in the case of a quadratic phase).  If $a,q$ are such that $(a,q) = 1$ and $|\theta - a/q| \le 1/q^2$, then 
\begin{equation}\label{weyl-usual} \big|\mb{E}_{x \in [X]} e(\theta x^2) \big| \ll (\log X) \big( q^{-1/2} + X^{-1/2} + q^{1/2} X^{-1}\big). \end{equation}
See, e.g. \cite[Lemma 3.1]{davenport}. In the quadratic case, the factor $X^{o(1)}$ present in that reference can be replaced by $O(\log X)$, since the use of the divisor bound at \cite[Equation (3.3)]{davenport} can be replaced by a trivial bound. (This tighter bound is not at all essential for our argument.)

Now set $Q := c(\delta/\log X)^{2} X^2$ for some small $c > 0$. By Dirichlet's theorem, there is some $q$, $1 \le q \le Q$, such that $|\theta - a/q| \le 1/qQ \le 1/q^2$. Applying \cref{weyl-usual}, for an appropriate $c$ the term $O(q^{1/2} X^{-1} \log X)$ is $\le \delta X/4$, and the term $O(X^{-1/2} \log X)$ satisfies the same bound by the assumption on $\delta$. It follows that $\delta X/2 \ll (\log X) q^{-1/2}$, so $q$ satisfies the required bound. Noting that $\Vert \theta q \Vert_{\R/\Z} \le Q^{-1}$ completes the proof.
\end{proof}

We now define our smoothly-weighted version of the squares.
\begin{definition}\label{def:weight-square}
Let $w$ be as in \cref{lem:bump}. Define $g_{X,\Box}:\Z \rightarrow \R$ via
\[g_{X,\Box}(x) := \left\{ \begin{array}{ll}
0 & \text{ if $|x|$ is not a square }\\
t w(t^2/X) & \text{ if } |x| = t^2\text{ with }t\in \N.
\end{array}\right.\]
\end{definition}

Here, and in what follows, write
\[ \wh{g_{X,\Box}}(\theta) := \sum_{x}g_{X,\Box}(x)e(-\theta x). \]
The estimates we need concern minor- and major- arc values of $\theta$ respectively. First, the minor arc estimate.
\begin{lemma}\label{minor-arc}
Let $X$ be a sufficiently large parameter. Suppose that $\delta \ge X^{-1/10}$ and that $|\wh{g_{X,\Box}}(\theta)|  \ge \delta X$.
Then there exists a natural number $q\ll (\log X)^2 \delta^{-4}$ such that 
\[\snorm{q\theta}_{\R/\Z}\ll (\log X)^2 \delta^{-4} X^{-1}.\]
\end{lemma}
\begin{proof} In this proof $C$ denotes a sufficiently large constant which may depend on the choice of function $w$ in \cref{lem:bump} but is independent of all other parameters. 

Writing in the definition of $g_{X,\Box}$, we see that the assumption is
\[ 2\Big|  \Re  \sum_{1 \le t \le X^{1/2}}   t w \Big(\frac{t^2}{X}\Big) e(-\theta t^2)  \Big|\ge \delta X.\]
Expanding $w$ via Fourier inversion and applying the triangle inequality, we obtain
\[ \int^{\infty}_{-\infty} \big|\wh{w}(\xi)\big|  \Big| \sum_{1 \le t \le X^{1/2}} 2t e \Big(\Big(-\theta + \frac{\xi}{X}\Big) t^2 \Big) \Big|d\xi \ge \delta X.\] By \cref{lem:bump} and the trivial bound of $2X$ for the inner exponential sum, the contribution from $|\xi| \ge C/\delta$ is at most $\delta X/2$ (a much weaker estimate on $\wh{w}$ than the one in \cref{lem:bump} would suffice here). Since $\int |\wh{w}(\xi)| d\xi \le C$, it follows that there is some $\xi$ with $|\xi| \le C/\delta$ such that 
\[ \Big| \sum_{1 \le t \le X^{1/2}} 2t e \Big(\Big(-\theta + \frac{\xi}{X}\Big) t^2 \Big) \Big| \ge \frac{\delta X}{2C}.\] By partial summation this may be written
\[ 2 X^{1/2} \Big| \int^1_0 dy  \sum_{y X^{1/2} \le t \le X^{1/2}}  e \Big(\Big(-\theta + \frac{\xi}{X}\Big) t^2 \Big) \Big| \ge \frac{\delta X}{2C}, \] so for some $y \in [0,1]$ we have
\[ \Big| \sum_{y X^{1/2} \le t \le X^{1/2}} e \Big(\Big(-\theta + \frac{\xi}{X}\Big) t^2 \Big) \Big| \ge \frac{\delta X^{1/2}}{4C}.\] Thus, by the triangle inequality, there is some $N \le X^{1/2}$ (we can take either $N = X^{1/2}$ or $N = y X^{1/2}$) such that 
\begin{equation}\label{pre-weyl} \Big| \sum_{t \le N} e \Big(\Big(-\theta + \frac{\xi}{X}\Big) t^2 \Big) \Big| \ge \frac{\delta X^{1/2}}{8C} \ge \frac{\delta N}{8C}.\end{equation}
We now apply Weyl's inequality, \cref{weyl-consequence}. The condition $\delta/8C \ge N^{-1/3}$ required for this application follows from the assumption $\delta \ge X^{-1/10}$, the fact that $N \ge \delta X^{1/2}/8C$ (which follows from the left-hand inequality in \cref{pre-weyl}) and the fact that $X$ is sufficiently large. We conclude that there is some $q \ll (\log X)^2\delta^{-2}$ such that 
\[ \Big\Vert q \Big(-\theta + \frac{\xi}{X} \Big) \Big\Vert_{\R/\Z} \ll (\log X)^2 \delta^{-2} N^{-2} \ll (\log X)^2 \delta^{-4} X^{-1}.\] From the triangle inequality and the fact that $|\xi| \le C/\delta$, this implies
\[ \Vert q \theta \Vert_{\R/\Z} \ll (\log X)^2 \delta^{-4} X^{-1},\] which is the desired result.
\end{proof}

Now we turn to the major arc estimate.
\begin{lemma}\label{major-arc}
Let $X$ be a sufficiently large parameter. Let $q$ be a positive integer with $q \le X^{1/8}$ and suppose that $a$ is coprime to $q$. Suppose that $|\theta| \le X^{-7/8}$. Then 
\[\Big|\wh{g_{X,\Box}}\Big(\frac{a}{q} + \theta \Big)\Big|\ll Xq^{-1/2} e^{-|\theta X|^{1/2}}+ X^{3/4}.\]
\end{lemma}
\begin{proof}
Let $Y \ge 1$ be a real parameter. If $\psi : [0,1] \rightarrow \C$ is continuous on $[0,1]$ and differentiable on $(0,1)$ and if $1 \le q \le Y$ then we have the bound
\begin{equation}\label{sum-integral}  \Big| \frac{q}{Y} \sum_{\substack{1 \le n \le Y \\ n \equiv b \mdsub{q}}} \psi\Big(\frac{n}{Y}\Big) - \int^1_0 \psi(x) dx \Big| \ll \frac{q}{Y} \max( \Vert \psi \Vert_{\infty}, \Vert \psi' \Vert_{\infty}).\end{equation}
To see this, first note that by the mean value theorem and matching summands
\[ \Big|  \sum_{\substack{1 \le n \le Y \\ n \equiv b \mdsub{q}}} \psi\Big(\frac{n}{Y}\Big)- \sum_{\substack{1 \le n \le Y \\ n \equiv b' \mdsub{q}}} \psi\Big(\frac{n}{Y}\Big)\Big| \ll \Vert \psi \Vert_{\infty} + \Vert \psi' \Vert_{\infty}\] for any $b, b'$ (where here the $\Vert \psi \Vert_{\infty}$ term comes from $n < q$ and $n > Y - q$, which may be present in one sum but not the other). Thus, by averaging it is enough to handle the case $q = 1$, which is the standard comparison between sums and integrals.

Let $F : [0,1] \rightarrow \C$ be continuous on $[0,1]$ and differentiable on $(0,1)$. Applying \cref{sum-integral} with $\psi(x) = 2x F(x^2)$ and making a substitution in the integral gives
\[ \Big| \frac{q}{Y} \sum_{\substack{1 \le n \le Y \\ n \equiv b \mdsub{q}}} \frac{2n}{Y} F\Big(\frac{n^2}{Y^2}\Big) - \int^1_0 F(x) dx \Big| \ll \frac{q}{Y} \max\Big( \Vert F \Vert_{\infty}, \Vert F' \Vert_{\infty}\Big).\]
Applying this with $F(x) := w(x) e(-\xi x)$ yields
\begin{equation}\label{to-use-bq}  \Big| \frac{q}{Y} \sum_{\substack{1 \le n \le Y \\ n \equiv b \mdsub{q}}} \frac{2n}{Y} w\Big( \frac{n^2}{Y^2} \Big) e\Big({-\frac{\xi n^2}{Y^2}}\Big) - \wh{w}(\xi) \Big| \ll \frac{q |\xi|}{Y} \max\big( \Vert w \Vert_{\infty}, \Vert w' \Vert_{\infty}\big) \ll \frac{q |\xi|}{Y}.\end{equation}
As in the proof of \cref{minor-arc},
\[ \wh{g_{X, \Box}}\Big(\frac{a}{q} + \theta\Big) = 2 \Re \sum_{1 \le n \le X^{1/2}}   n w \Big(\frac{n^2}{X}\Big) e\Big({-\Big(\frac{a}{q} + \theta\Big) n^2}\Big).\] Setting $Y := X^{1/2}$ and $\xi := \theta X$ and splitting into residue classes modulo $q$, we obtain
\[ \wh{g_{X, \Box}}\Big(\frac{a}{q} + \theta\Big) = \Re  \sum_{b \mdsub{q}} e\Big({-\frac{ab^2}{q}}\Big) \sum_{\substack{1 \le n \le Y \\ n \equiv b \mdsub{q}}}  2n  w \Big(\frac{n^2}{Y^2}\Big) e\Big({-\frac{\xi n^2}{Y^2}}\Big).\] By \cref{to-use-bq},
\[ \wh{g_{X, \Box}}\Big(\frac{a}{q} + \theta\Big) = \frac{X}{q}\Re \Big(\wh{w}(\xi)  \sum_{b \mdsub{q}} e\Big({-\frac{ab^2}{q}}\Big) \Big) + O(X^{1/2} q |\xi|).\] 
To conclude, we apply the standard Gauss sum estimate
\[\Big| \sum_{b \mdsub{q}} e\Big({-\frac{ab^2}{q}}\Big) \Big| \ll \sqrt{q}, \]
which may be found (for example) as \cite[Lemma 6.4]{davenport} and the estimate on $\wh{w}(\xi)$ from \cref{lem:bump}.
\end{proof}

\section{Density increment arguments}\label{sec6}

As explained in the introduction, we will use the density increment strategy common to previous works on the problem.  Suppose that $A \subset [X]$ is a set with density $\alpha$ and no square difference. We will establish a variety of different types of density increment, which we summarise in the following proposition.

Constants denoted by $c, c_0, C$ are absolute.

\begin{proposition}\label{density-increment-calc} There are constants $c, C$ with $0 < c < \frac{1}{10}$ and $C > 10$ such that the following holds. Let $X \ge c^{-C}$ and suppose that $X^{-1/100} < \alpha \le \frac{2}{3}$. Suppose that $A \subset [X]$ is a set with density $\alpha$ and no square difference. Then in all cases
\begin{enumerate}
    \item[\textup{($\ast$)}] there is a subprogression of $[X]$ with square common difference and length $\ge \alpha^C X$ on which the density of $A$ is at least $\alpha + \alpha^C$.
\end{enumerate} Moreover, if $\alpha \le c$ then at least one of the following options holds:
 \begin{enumerate}
\item[\textup{(1)}] $\alpha  \le \exp(- c \sqrt{\log X})$; 
\item[\textup{(2)}] for some $m \in \N$, there is a subprogression $P$ with square common difference and length $\ge \alpha^{Cm} X$ on which the density of $A$ is at least $2^{m} \alpha$;
\item[\textup{(3)}] there is a subprogression $P$ with square common difference and length $\ge \alpha^{C} X$ on which the density of $A$ is at least $(1 + c)\alpha$.
\end{enumerate}
\end{proposition}
\begin{remarks} The subprogressions in ($\ast$), (2) and (3) are all proper (that is, not equal to $[X]$) since the density of $A$ on them is strictly larger than $\alpha$ in each case. We could easily combine (2) and (3) into one statement (simply by replacing the constant $2$ in (2) by $1 + c$) and in fact we will do this in the proof that \cref{density-increment-calc} implies \cref{thm:main} below. We have left them separate in the statement of \cref{density-increment-calc} to emphasise that they come from slightly different sources, as will become apparent later.
\end{remarks}

The proof of this proposition will occupy most of the rest of the paper. We now show that it implies \cref{thm:main}. 

\begin{proof}[Proof of \cref{thm:main}, assuming \cref{density-increment-calc}] Recall that $s(X)$ denotes the size of the largest subset of $[X]$ containing no square difference. 
For notational brevity, in the following proof it is convenient to write 
\begin{equation}\label{FX-def} F(X) := (\log X)^{1/2} \end{equation} for $X \ge 1$.
Suppose that we have the statement of \cref{density-increment-calc}. We will prove, by strong induction on $X$, that for a suitable positive constant $c_0$ (which will depend on $c, C$) and for all $X \ge 1$ we have 
\begin{equation}\label{sx-bd} s(X) \le X \exp(-c_0 F(X)),\end{equation} which is the statement of \cref{thm:main}. If $1 \le X \le c^{-C}$, this statement is trivial if one takes $c_0$ small enough, so suppose from now on that $X \ge c^{-C}$ and that we have established \cref{sx-bd} for all $X' < X$. 

In several places in what follows, we will need the bound
\begin{equation}\label{f-der-bd}  F(X) - F(e^{-h} X) \le h (\log X)^{-1/2}  \end{equation}
for $h \le \log X$. To see this, note that $F(X) - F(e^{-h} X) = (\log X)^{1/2} - (\log X - h)^{1/2}$. For $0 \le h \le t$ we have
\[  t^{1/2} - (t - h)^{1/2}  =  t^{1/2} \Big( 1 - \Big(1 - \frac{h}{t} \Big)^{1/2} \Big) \le h t^{-1/2} ,\]
where the last step follows from $(1 - x)^{1/2} \ge 1 - x$ for $0 \le x \le 1$. From this, the bound \cref{f-der-bd} follows.

Let $A \subset [X]$ be a set of density $\alpha = s(X)/X$ such that $A$ contains no square difference.
The aim is to show that 
\begin{equation}\label{dens-aim} \alpha \le \exp(-c_0 F(X)) .\end{equation}
In the following discussion, $P$ will always denote a subprogression of $[X]$ with square common difference, length $X' < X$ and on which the density of $A$ is $\alpha' > \alpha$.

Suppose first that $\alpha > c$. Observe that $\alpha \le 2/3$ as any $A\subseteq [X]$ with $|A|>2X/3$ contains a pair of consecutive elements and hence a square difference. Then as $\alpha\in (c,2/3]$ we may apply \cref{density-increment-calc} ($\ast$) to find $P$ with $X' \ge \alpha^C X \ge e^{-h}X$ and $\alpha' \ge \alpha + \alpha^C \ge (1 + \eps) \alpha$, where $h := C \log(1/c)$ and $\eps := 1 + c^{C-1}$. By the inductive hypothesis, 
\[ (1 + \eps) \alpha \le \alpha' \le \exp(-c_0 F(X')) \le \exp(-c_0 F(e^{-h} X)).\] We wish to deduce \cref{dens-aim} from this, which is equivalent to showing
\[ F(X) - F(e^{-h} X) \le \frac{\log(1 + \eps)}{c_0}\] By \cref{f-der-bd} this holds provided that $c_0$ is chosen sufficiently small. Note here that the condition $h \le \log X$ required for the application of \cref{f-der-bd} is satisfied because of the assumption that $X \ge c^{-C}$.

Suppose from now on that $\alpha \le c$. We may assume that $\alpha > X^{-1/100}$, since otherwise \cref{dens-aim} follows immediately. Then one of options \cref{density-increment-calc} (1) -- (3) holds. If (1) holds, we are done provided that $c_0 \le c$ (which we may assume). It remains to treat the cases (2) and (3). We can combine these into one by asserting that there is $P$ with length $X' \ge \alpha^{Cm}X$ on which the density of $A$ is at least $(1 + c)^m$. Set $h := Cm \log(1/\alpha)$. Since the density of $A$ on $P$ cannot exceed $1$, we have $m \le \frac{\log(1/\alpha)}{\log(1 + c)} \le \frac{2}{c} \log(1/\alpha)$ and so $h \le \frac{2C}{c} \log(1/\alpha)^{2}$. If $c_0 \le (c/2C)^{1/2}$, this implies that $h \le \log X$ unless $\alpha \le \exp(-c_0 (\log X)^{1/2})$, that is to say unless \cref{dens-aim} holds. We may therefore assume in what follows that $h \le \log X$.

By the inductive hypothesis,
\[ (1 + c)^m \alpha \le \exp(-c_0 F(X')) \le \exp (-c_0 F(e^{-h} X)) ,\] with $h$ as above. Thus \eqref{dens-aim} follows if we can show that
\begin{equation}\label{to-show-2} F(X) - F(e^{-h} X) \le \frac{m \log (1 + c)}{c_0}.\end{equation}  By \cref{f-der-bd},  it is enough to show
\[ \log(1/\alpha) \cdot (\log X)^{-1/2}  \le \frac{\log (1 + c)}{c_0 C}.\]
If $c_0 \le \big(\frac{\log(1 + c)}{C}\big)^{1/2}$ then this holds unless $\alpha \le \exp(-c_0 (\log X)^{1/2})$. Either way, \cref{dens-aim} holds.\vspace*{8pt}
\end{proof}

\subsection{Fourier coefficients and density increments}

Let $A \subseteq [X]$ be a nonempty set of density $\alpha$ with balanced function $f_A := \mathbf{1}_A - \alpha 1_{[X]}$. In this section we recall some results stating that large Fourier coefficients of $f_A$ allow one to find a long progression on which the density of $A$ is increased. 

\begin{lemma}\label{dens-increment-lem} Let $\eta \in (0,1)$. Let $q \ge 1$ be an integer, and let $\xi \in \R$. Set $T := \max(1, |\xi|X)$. Suppose that either 
\begin{equation}\label{roth-increment} \Big|\wh{f_A}\Big(\frac{a}{q} + \xi \Big)\Big| \ge \eta \alpha X\end{equation} for some $a \in \Z$ or that
\begin{equation}\label{szem-condition} \sum_{a \mdsublem{q}}\Big|\wh{f_{A}}\Big(\frac{a}{q} + \xi \Big)\Big|^2\ge \eta \alpha^2 X^2 .\end{equation}
Then there is a progression $P \subseteq [X]$ with common difference $q^2$ and length $\gg q^{-2}T^{-1} \eta X$ on which the density of $A$ is at least $(1 + \eta/20) \alpha$.
\end{lemma}
\begin{proof}
If $q^2 T > 2^{-10}\eta X$ then the result is trivial (just take $P$ to be a progression of length $1$ consisting of a point of $A$). Suppose henceforth that
\begin{equation}\label{not-triv-assump} q^2 T \le 2^{-10}\eta X.\end{equation}
We now handle the two conditions \cref{roth-increment} and \cref{szem-condition} in turn, beginning with the former. In this case the lemma states the type of density increment used by S\'ark\"ozy in his original paper \cite{sarkozy-i}. Foliating into progressions to modulus $q$, \cref{roth-increment} implies that 
\begin{equation}\label{altszem-trick-2} \sum_{b \in \Z/q\Z} \Big| \sum_{\substack{ n \in [X] \\ n \equiv b \mdsub{q}}} e(-\xi n) f_A(n) \Big| \ge \eta \alpha X. \end{equation}
Consider any partitioning of $[X]$ into intervals of the following type: for each $b$, partition $\{ n \in [X] : n \equiv b \md{q}\}$ into $k_b$ progressions $I_{j,b}$, $j = 1,\dots, k_b$ of common difference $q^2$ and diameter $\le \frac{1}{32}T^{-1} \eta X$. Applying the triangle inequality to \cref{altszem-trick-2} gives
\begin{equation}\label{altszem-trick-3} \sum_{b \in \Z/q\Z}\sum_{j = 1}^{k_b} \Big| \sum_{n \in I_{j,b}} e(-\xi n) f_A(n) \Big| \ge \eta \alpha X. \end{equation}
For each $j,b$ let $n_{j,b}$ be an arbitrary point of $I_{j,b}$. The error in replacing each $e(\xi n)$ by $e(\xi n_{j,b})$ is at most 
\[  \sup_{j,b}\sup_{n \in I_{j,b}} | e(-\xi n) - e(-\xi n_{j,b}) |\sum_{b \in \Z/q\Z}\sum_{j = 1}^{k_b}  \sum_{n \in I_{j,b}} |f_A(n)|  \le 2\pi | \xi | (\frac{1}{32} T^{-1} \eta X)  \sum_{n \in [X]} |f_A(n)| < \frac{1}{2}\eta\alpha X.\]
It follows from this and \cref{altszem-trick-3} that
\[ \sum_{b \in \Z/q\Z} \sum_{j = 1}^{k_b} \Big| \sum_{n \in I_{j,b}} f_A(n) \Big| \ge \frac{1}{2}\eta \alpha X.\] 
Thus, since $f_A(n)$ sums to $0$ on $[X]$, we have
\[ \sum_{b \in \Z/q\Z} \sum_{j = 1}^{k_b}\Big(  \sum_{n \in I_{j,b}} f_A(n) + \Big| \sum_{n \in I_{j,b}} f_A(n) \Big|\Big) \ge \frac{1}{2}\eta \alpha X.\]
Therefore, by pigeonhole there is some $I_{j,b}$ on which
\[  \sum_{n \in I_{j,b}} f_A(n) + \Big| \sum_{n \in I_{j,b}} f_A(n) \Big| \ge \frac{1}{2}\eta \alpha |I_{j,b}|\] and hence
\begin{equation}\label{dens-inc-7} \sum_{n \in I_{j,b}} f_A(n) \ge \frac{1}{4}\eta \alpha |I_{j,b}|.\end{equation}
Now we specify a partitioning of the above form. First, partition $[X]$ into intervals of length between $\frac{1}{64} T^{-1} \eta X$ and $\frac{1}{32} T^{-1} \eta X$; on each one, and for each $b' \in \Z/q^2 \Z$, the set of elements with $n \equiv b' \md{q^2}$ is a progression of diameter $\le \frac{1}{32} T^{-1} \eta X$ and length $\gg q^{-2} T^{-1} \eta X$. The possibility of this partitioning and the last assertion both follow from \cref{not-triv-assump}. From \cref{dens-inc-7}, it follows that the density of $A$ on at least one of these progressions is at least $(1 + \eta/4)\alpha$, and this completes the analysis of \cref{roth-increment}.

We now turn to the analysis of \cref{szem-condition}. We remark that an equivalent statement to this is implicit in \cite{pintz-steiger-szemeredi}. One may check by expanding and orthogonality of characters that is equal to
\begin{equation}\label{top-p-19} \sum_{a\mdsub{q}}\Big|\wh{f_{A}}\Big(\frac{a}{q} + \xi\Big)\Big|^2 = q\sum_{b \in \Z/q\Z} \Big| \sum_{\substack{ n \in [X] \\ n \equiv b \mdsub{q}}} e(-\xi n) f_A(n) \Big|^2, \end{equation}
 so from the assumption \cref{szem-condition} we have
 \begin{equation}\label{altszem-trick} \sum_{b \in \Z/q\Z} \Big| \sum_{\substack{ n \in [X] \\ n \equiv b \mdsub{q}}} e(-\xi n) f_A(n) \Big|^2 \ge q^{-1}\eta \alpha^2  X^2. \end{equation}
Suppose there is $b$ such that 
\begin{equation}\label{b-big}\Big|\sum_{n\equiv b\mdsub{q}}e(-\xi n)f_A(n)\Big|\ge 5q^{-1}\alpha X.\end{equation}
Since
\[\Big|\sum_{n\equiv b\mdsub{q}}e(-\xi n)f_A(n)\Big|\le q^{-1}\alpha X + 1 + \sum_{n\equiv b\mdsub{q}}\mathbf{1}_A(n),\]
it follows that 
\[ \sum_{n\equiv b\mdsub{q}}\mathbf{1}_A(n) \ge 3 q^{-1} \alpha X,\] and so $A$ has density at least $2\alpha$ on the progression $\{n \in [X] :n \equiv b \md{q}\}$. By a trivial averaging, $A$ also has density at least $2\alpha$ on some progression $\{n \in [X] : n \equiv b' \md{q^2}\}$, which certainly implies the result. Suppose then that \cref{b-big} does not hold for any $b$. Then \cref{altszem-trick} implies that
\[ \sum_{b \in \Z/q\Z} \Big| \sum_{\substack{ n \in [X] \\ n \equiv b \mdsub{q}}} e(-\xi n) f_A(n) \Big| \ge \frac{1}{5}\eta \alpha X. \]
This is exactly \cref{altszem-trick-2} (albeit with $\eta$ replaced by $\eta/5$) and we may conclude as in the analysis of \cref{roth-increment}.
\end{proof}

\subsection{First steps towards the density increment}\label{sec43}

We now proceed towards the proof of \cref{density-increment-calc}, which we will complete in the next section.
In what follows, $C_1$ denotes an absolute constant whose exact nature arises from the implicit constants in \cref{sec5}; the conditions it needs to satisfy will be specified during the course of the proof of \cref{g-fourier} below. We imagine that this constant is fixed, once and for all.

The aim is then to show that if further constants $c, C$ with $0 < c < 1 < C$ are taken then, if $1/c$ and $C$ are sufficiently large, all the statements in \cref{density-increment-calc} are true. 

With this in mind, suppose in what follows that $A \subset [X]$ is a set with no square difference and density $\alpha$. Assume throughout that $X \ge c^{-C}$ and that $X^{-1/100} < \alpha \le \frac{2}{3}$, as in the statement of \cref{density-increment-calc}. In particular, we may always assume that $X$ is sufficiently large, and we will do so without further comment. For notational brevity, we write $L := \log(1/\alpha)$. 

We define the major arcs $\mf{M} \subset \R/\Z$ by 
\begin{equation}\label{major-def} \mf{M}  := \bigcup_{q \le C_1\alpha^{-2}}\bigcup_{\substack{1 \le a < q \\ (a,q) = 1}} \Big[ \frac{a}{q} - \tau, \frac{a}{q} + \tau\Big], \quad \mbox{where} \quad \tau := C_1 L^2 X^{-1}.\end{equation} Set $\mf{m}  := (\R/\Z) \setminus \mf{M}$.
Note that $\mf{M}$ and $\tau$ both depend on $\alpha$, but we do not indicate this explicitly

\begin{lemma}\label{g-fourier} We have $\sup_{\xi \in \mf{m}}|\wh{g_{X,\Box}}(\xi)|\le 2^{-9} \alpha X$.
\end{lemma}
\begin{proof}
It is convenient to introduce some auxilliary `wide' major arcs $\tilde{\mf{M}} \subset \R/\Z$ by 
\[  \tilde{\mf{M}}  := \bigcup_{q \ll (\log X)^2 \alpha^{-4}}\bigcup_{\substack{1 \le a < q \\ (a,q) = 1}} \Big[ \frac{a}{q} - \tilde{\tau}, \frac{a}{q} + \tilde{\tau}\Big] , \quad \mbox{where} \quad \tilde{\tau}  := C_1(\log X)^4\alpha^{-8} X^{-1}. \]  Suppose that $|\wh{g_{X,\Box}}(\xi)|\ge 2^{-9}\alpha X$. By \cref{minor-arc}, there is $q \ll (\log X)^2 \alpha^{-4}$ such that $\Vert q \xi \Vert_{\R/\Z} \ll (\log X)^2 \alpha^{-4} X^{-1}$. If $C_1$ is chosen appropriately large, this means that $\xi \in \tilde{\mf{M}}$.

To analyse $\xi \in \tilde{\mf{M}}$ further, we use \cref{major-arc}. By the previous paragraph, there is some $a \in \Z$ such that if we write $\theta = \xi - a/q$ then $|\theta| \le \tilde{\tau}$. Applying \cref{major-arc} with this $\theta$, the required condition $|\theta| \le X^{-7/8}$ is satisfied by the definition of $\tilde\tau$ and the assumption that $\alpha > X^{-1/100}$ (and since $X$ is sufficiently large in terms of $C_1$). We conclude that 
\[ \alpha X \ll X q^{-1/2} e^{-|\theta X|^{1/2}} + X^{3/4}.\]
The $X^{3/4}$ term can be ignored since $\alpha > X^{-1/100}$, so we obtain $\alpha \ll q^{-1/2} e^{-|\theta X|^{1/2}}$, from which we deduce that both $q \le C_1 \alpha^{-2}$ (if $C_1$ is chosen appropriately) and that $\alpha \ll e^{-|\theta X|^{1/2}}$, and thus $|\theta| \ll L^2 X^{-1}$. Thus (again, if $C_1$ is chosen appropriately) $\xi \in \mf{M}$, which is what we wanted to prove.
\end{proof}

\begin{remarks}
At this point, all the conditions $C_1$ needed to satisfy have been specified. Introducing the two types of major arcs (or an equivalent device) is a well-known technique in the circle method, used to deduce a `log-free' variant of Weyl's inequality from the usual one, which is essentially what has been accomplished here. See for instance \cite{wooley}. In our setting it also has the advantage that we can exploit the very smooth weight $w$ and hence take the major arcs $\mf{M}$ to be rather narrow. 
\end{remarks}

Now we turn to the main analysis. Observe that, from the assumption that $A$ has no square difference,
\begin{equation}\label{no-square-diff} \sum_{x,y}\mathbf{1}_{A}(x)\mathbf{1}_{A}(y)g_{X,\Box}(x-y) = 0.\end{equation}
We may assume that
\begin{equation}\label{we-may-assume} \sum_{x, y} \mathbf{1}_A(x) \mathbf{1}_{[X]} (y) g_{X, \Box}(x - y) \ge 2^{-8} \alpha X^2.\end{equation} 
To see this, note that $x - y$ is supported on $\pm t^2$, $t^2 \le X$, by the definition of $g_{X,\Box}$. Thus if \cref{we-may-assume} fails, we have
\[ \sum_{x}\sum_{1 \le t \le X^{1/2}} \mathbf{1}_A(x) 1_{[X]} (x - t^2) t w(t^2/X) \le 2^{-8} \alpha X^2\] (or an essentially identical case with $+ t^2$, which we omit). That is, 
\[ \sum_{1 \le t \le X^{1/2}} t w(t^2/X) \big| A \cap [X - t^2]\big| \le 2^{-8} \alpha X^2.\] In particular, by the support properties of $w$ (given in \cref{lem:bump}) it follows that 
\[ \sum_{X^{1/2}/4 \le t \le X^{1/2} /2} t \big| A \cap [X - t^2]\big| \le 2^{-7} \alpha X^2, \] whence there is some $t \in [\frac{1}{4} X^{1/2}, \frac{1}{2} X^{1/2}]$ such that $|A \cap [X - t^2]| \le \frac{1}{4} \alpha X \le \frac{1}{2}\alpha (X - t^2)$, which implies (after a short calculation) that $|A \cap (X - t^2 , X]| \ge \frac{3}{2}\alpha t^2$. That is, $A$ has density at least $3\alpha/2$ on the interval $(X - t^2, X]$, which has length at least $3X/4$. This is an enormous density increment for $A$, much better than both \cref{density-increment-calc} ($\ast$) and \cref{density-increment-calc} (3).

Therefore \cref{we-may-assume} does hold. Writing $f_A := \mathbf{1}_A - \alpha \mathbf{1}_{[X]}$ for the balanced function of $A$, it follows from \cref{no-square-diff} and \cref{we-may-assume} that we have
\[\Big|\sum_{x,y}\mathbf{1}_{A}(x) f_A(y) g_{X,\Box}(x-y)\Big|\ge 2^{-8}\alpha^2 X^2.\]
Taking Fourier transforms, this implies that 
\begin{equation}\label{no-squares-fourier}\Big|\int_{\R/\Z}\wh{\mathbf{1}_A}(-\xi) \wh{f_A}(\xi) \wh{g_{X,\Box}}(\xi) d\xi\Big|\ge 2^{-8} \alpha^2 X^2.\end{equation}

By Cauchy-Schwarz, Parseval and \cref{g-fourier} we have
\begin{align*}
\Big| \int_{\mf{m}}\wh{\mathbf{1}_A}(-\xi) \wh{f_A}(\xi) \wh{g_{X,\Box}}(\xi) d\xi\Big| & \le \sup_{\xi \in \mf{m}}|\wh{g_{X,\Box}}(\xi)| \Big(\int_{\R/\Z}  |\wh{\mathbf{1}_A}(\xi)|^2 d\xi    \Big)^{1/2} \Big(\int_{\R/\Z}  |\wh{f_A}(\xi)|^2 d\xi    \Big)^{1/2} \\ & \le  2^{-9}\alpha^2 X^2/2, 
\end{align*} so from \cref{no-squares-fourier} we have
\begin{equation}\label{no-squares-fourier-2}\Big|\int_{\mf{M}}\wh{\mathbf{1}_A}(\xi) \wh{f_A}(-\xi) \wh{g_{X,\Box}}(\xi) d\xi\Big|\ge 2^{-9} \alpha^2 X^2.\end{equation}
Write 
\begin{equation}\label{m-prime}  \mf{M} := [-\tau, \tau] \cup \mf{M}', \quad 
\mbox{where} \quad \mf{M}' := \bigcup_{2 \le q \le C_1 \alpha^{-2}}\bigcup_{\substack{1 \le a < q \\ (a,q) = 1}} \Big[ \frac{a}{q} - \tau, \frac{a}{q} + \tau\Big]. \end{equation} Here $\tau = C_1L^2 X^{-1}$ is as in the definition \cref{major-def} of the major arcs.
From \cref{no-squares-fourier-2} we have either 
\begin{equation}\label{case1}
\Big|\int^{\tau}_{-\tau}\wh{\mathbf{1}_A}(\xi) \wh{f_A}(-\xi) \wh{g_{X,\Box}}(\xi) d\xi\Big|\ge 2^{-10} \alpha^2 X^2
    \end{equation}
or 
\begin{equation}\label{case2}
\Big|\int_{\mf{M}'}\wh{\mathbf{1}_A}(\xi) \wh{f_A}(-\xi) \wh{g_{X,\Box}}(\xi) d\xi\Big|\ge 2^{-10}\alpha^2 X^2.\end{equation}
We handle these two cases slightly differently.\vspace*{8pt}

\emph{Case 1:} \cref{case1} holds. Using the trivial bound $|\wh{1_A}(\xi)| \le \alpha X$, it follows from \cref{case1} that
\begin{equation}\label{case1-a} \sup_{|\xi| \le \tau} |\wh{f_A}(\xi)| \int^{\tau}_{-\tau} |\wh{g_{X,\Box}}(\xi)| d\xi \ge 2^{-10}\alpha X.\end{equation}
Now from \cref{major-arc} with $(a,q) = (0,1)$ (noting from the definition \cref{major-def} that $\tau \le X^{-7/8}$ if $X$ is sufficiently large) we have the estimate $|\wh{g_{X,\Box}}(\xi)| \ll X e^{-| \xi X |^{1/2}} + X^{3/4}$ and so 
\[\int^{\tau}_{-\tau} |\wh{g_{X,\Box}}(\xi)| d\xi \ll 1. \] Therefore \cref{case1-a} implies that there is some $\xi$, $|\xi| \le \tau$, such that 
\[ |\wh{f_A}(\xi)| = \Big| \sum_{x \in [X]} f_A(x) e(-\xi x) \Big| \gg \alpha X.\]
To this conclusion, we apply \cref{dens-increment-lem}. Recalling the definition \cref{major-def} of $\tau$, we may take $(a,q) = (0,1)$ and $T \ll L^2$ in that lemma. We conclude that $A$ has density at least $\alpha(1 + c)$ for some absolute $c > 0$ on some progression with common difference $1$ and length $\gg L^{-2} X$. This is a strong density increment, superior to \cref{density-increment-calc} ($\ast$) and \cref{density-increment-calc} (3).\vspace*{8pt}

\emph{Case 2:} \cref{case2} holds. To put ourselves in a position to apply \cref{thm:level-d-int}, it is convenient to replace the balanced function $f_A$ by $\mathbf{1}_A$. For this, we make the simple observation that
\begin{equation}\label{m-prime-balanced} \sup_{\xi \in \mf{M}'} |\wh{f_A}(\xi) - \wh{\mathbf{1}_A} (\xi)| \ll \sup_{\xi \in \mf{M}'} |\wh{\mathbf{1}_{[X]}}(\xi)| \ll X^{-1/2},\end{equation} which follows from evaluating $\wh{\mathbf{1}_{[X]}}(\xi)$ by summing a geometric series, together with the fact that $\Vert \xi \Vert_{\R/\Z} \gg \alpha^{2} \gg X^{-1/2}$ for $\xi \in \mf{M}'$.

Thus, $f_A$ can be replaced by $\mathbf{1}_A$ in \cref{case2} with minuscule error. This, the fact that $\wh{\mathbf{1}_A}(-\xi) = \overline{\wh{\mathbf{1}_A}(\xi)}$, and the triangle inequality imply that 
\[
\int_{\mf{M}'}|\wh{\mathbf{1}_A}(\xi)|^2 \big| \wh{g_{X,\Box}}(\xi)\big| d\xi \ge 2^{-11} \alpha^2 X^2.\]
We estimate $|\wh{g_{X,\Box}}(\xi)|$ from above using \cref{major-arc}, splitting $\mf{M}'$ into the constituent intervals $[\frac{a}{q} - \tau, \frac{a}{q} + \tau]$. This gives
\[ \int^{\tau}_{-\tau} d\xi \sum_{2 \le q \le C_1 \alpha^{-2}}\sum_{\substack{1 \le a < q \\ (a,q) = 1}} \Big|\wh{\mathbf{1}_A}\Big(\frac{a}{q} + \xi\Big)\Big|^2 \Big(q^{-1/2} X e^{-|\xi X|^{1/2}} + X^{3/4} \Big) \gg \alpha^2 X^2. \]
The contribution of the $X^{3/4}$ term is negligible. Since $\int^{\tau}_{-\tau} Xe^{-|\xi X|^{1/2}}d\xi \ll 1$, we obtain that there exists some $\xi$, $|\xi| \le \tau = C_1 L^2 X^{-1}$, such that
\begin{equation}\label{just-before-thm12}  \sum_{2 \le q \le C_1\alpha^{-2}} q^{-1/2}\sum_{\substack{1 \le a < q \\ (a,q) = 1}}\Big|\wh{\mathbf{1}_A}\Big(\frac{a}{q} + \xi\Big)\Big|^2  \gg \alpha^2  X^2.\end{equation} 
The treatment up to this point is rather close to that of S\'ark\"ozy \cite{sarkozy-i}, albeit with an additional smoothing. From \cref{just-before-thm12}, S\'ark\"ozy proceeds to extract a large Fourier coefficient of $A$. So as to handle the case of very large $\alpha$, we will also give this argument, though below we will give a significantly more efficient argument for smaller $\alpha$. (An alternative here would have been to simply quote the main result of \cite{Fur77} or \cite{sarkozy-i} that $s(X) = o(X)$, but since we have already given 99\% of a proof of that statement, this seems unsatisfactory.)

By pigeonhole applied to \cref{just-before-thm12} there is some $a/q$ with $q \ge 2$ and $(a,q) = 1$ such that $|\wh{1_A} \big( \frac{a}{q} + \xi \big)| \gg \alpha^4 X$. By \cref{m-prime-balanced}, $|\wh{f_A} \big( \frac{a}{q} + \xi \big)| \gg \alpha^4 X$. (The double application of \cref{m-prime-balanced} could have been omitted in this part of the argument.) We apply \cref{dens-increment-lem}, taking $T \ll L^2$, which is permissible by the definition \cref{major-def} of major arcs. From this lemma it follows that $A$ has density $\ge \alpha + c'\alpha^4$ on some progression with common difference $q^2$ and length $\gg \alpha^8 L^{-2} X \gg \alpha^9 X^2$. Since $\alpha \le \frac{2}{3}$, for suitable $C$ the density of $A$ on this subprogression is $\ge \alpha + \alpha^C$, and the length of the subprogression is $\ge \alpha^C X$, which is \cref{density-increment-calc} ($\ast$).

\section{Applying the level \texorpdfstring{$d$}{} inequality: bounds for arithmetic energies}
\label{sec7-new}

Our task now is to obtain a density increment as in \cref{density-increment-calc} from \cref{just-before-thm12}. At the end of the last section we established \cref{density-increment-calc} ($\ast$) for all densities $\alpha \le \frac{2}{3}$. The remaining statements of \cref{density-increment-calc} concern the case where $\alpha \le c$ is sufficiently small, so we are free to assume this, and consequently that $L = \log(1/\alpha)$ is sufficiently large, in all that follows.

We will accomplish this task in the next two sections. Our main input will be \cref{thm:level-d-int}, and in this section we develop our main tool for applying it to \cref{just-before-thm12}, which is \cref{n-lem73} below.

The analysis involves a rather careful division of the set of $q$ according to properties of their prime factorizations, and we begin in the next subsection by setting up appropriate nomenclature.

\begin{remark} The reader may be interested to know that, although the final argument is not especially long, it was very far from obvious to us. Indeed, the set of exponents in our main theorem that we have at some point alighted upon during the development of this work is roughly $\{ c, \frac{3}{14}, \frac{1}{4} - o(1), \frac{3}{11},\frac{2}{7}, \frac{3}{10}, \frac{1}{3}-o(1), \frac{1}{3},\frac{2}{5},\frac{1}{2}-o(1), \frac{1}{2}\}$. A detailed proof of the exponent $\frac{1}{4} - o(1)$ may be found in the first arXiv version of the paper \cite{first-version}.
\end{remark}

\subsection{Basic setup and parameters} \label{sec8-1}

Suppose that $q \in \N$ has prime factorization $q = \prod_p p^{v_p}$. Then $q = q_* q'$ where $q_* := \prod_{p : v_p \ge 3} p^{v_p}$ and $q' := \prod_{p : v_p \le 2} p^{v_p}$. We call $q_*$ and $q'$ the cube-full and cube-free parts of $q$ respectively. Note that $(q_*, q') = 1$.

Observe for later use that 
\begin{equation}\label{38-bd} \sum_{t \;\cubefull}t^{-1/2}\le \prod_{p}\Big(1+\sum_{k\ge 3}\frac{1}{p^{k/2}}\Big)\ll 1;\end{equation} the fact that this inequality holds is the reason for using the notion of cube-free, rather than the more familiar notion of square-free, as we will see in the analysis that follows.

Throughout the next two sections, let $\WW$ be a sufficiently large absolute constant. How large it needs to be will be determined at the end of \cref{sec8.2}. Let $L = \log(1/\alpha)$ as usual. Given $q \in \N$ and its cube-free part $q'$:
\begin{itemize}
\item define $k(q')$ to be $1$ plus the number of distinct prime factors of $q'$ which are $> \WW L$;
\item let $d(q')$ be the product of all prime factors of $q'$  which are $\le \WW L$ (with multiplicity, which is $1$ or $2$ for each prime occurring);
\item let $p_2(q') < \cdots < p_k(q')$ be the prime factors of $q'$ which are $> \WW L$.
\end{itemize}

Write \begin{equation}\label{d-def} \mc{D} := (\log L)^{-1} \Z_{\ge 0}.\end{equation} For $t \in \N$, $k \in \Z_{\ge 0}$ and $\vecbeta = (\beta_1,\dots, \beta_k) \in \mc{D}^k$, define 
\[ \mathscr{Q}_{t,k,\vecbeta} := \{ q \in \N:  2 \le q \le C_1 \alpha^{-2} , \; q_* = t,\;  k(q') = k,\; d(q') \sim L^{\beta_1},\; p_i(q') \sim L^{\beta_i} \; \mbox{for $i = 2,\dots, k$}\}. \] 
Here, and in what follows, $x \sim L^{\beta}$ means that $L^{\beta} \le x < eL^{\beta} = L^{\beta + (\log L)^{-1}}$. Note that, since the $\beta_i$ range over $\mc{D}$, the sets $\mathscr{Q}_{t,k,\vecbeta}$ form a partition of the set of all integer $q$ with $2 \le q \le C_1 \alpha^{-2}$, that is to say the set of all $q$ in \cref{just-before-thm12}.
Observe that all but finitely many of the $\mathscr{Q}_{t,k,\vecbeta}$ are empty. More precisely, we may suppose throughout that
\begin{equation}\label{global-conds}  t \le C_1 \alpha^2 \quad \mbox{and} \quad k \le 2L/\log L. \end{equation}
The second bound here follows from the fact that $L$ is sufficiently large and the observation that if $k(q') = k$ then $q$ has $k$ distinct prime factors and so is certainly at least the product of the first $k$ primes, which is $e^{(1 + o(1)) k \log k}$ by the prime number theorem.
Denoting
\begin{equation}\label{beta-ordering} B_k := \{ \vecbeta \in \mc{D}^k : 0 \le \beta_1, \;  1 \le \beta_2 \le \beta_3 \le \cdots \le \beta_k, \quad \mbox{and} \quad \mbox{$L^{\beta_i} \le C_1 \alpha^{-2}$}\} \end{equation} we may assume that $\vecbeta \in B_k$ since, again, if this is not so then the corresponding set $\mathscr{Q}_{t,k, \vecbeta}$ is empty.

\subsection{An arithmetic energy bound}

Let $A \subset [X]$ be a set, as in the statement of \cref{dens-increment-lem}. Let $\xi \in \R$ be a real number with $|\xi| \le \tau = C_1 \alpha^{-2}$ (as for example in \cref{just-before-thm12}). Define
\begin{equation}\label{e-def} E_{t,k,\vecbeta} := \sum_{q \in \mathscr{Q}_{t,k, \vecbeta}} \sum_{\substack{1 \le a < q \\ (a,q) = 1}} \Big| \wh{1_A}\Big(\frac{a}{q} + \xi\Big) \Big|^2. \end{equation} We suppress the dependence on $\xi$ in the notation. The following estimate will be our main application of \cref{thm:level-d-int}.

\begin{lemma}\label{n-lem73}
Suppose that $1 \le r \le k$. If $r \ge 2$, suppose additionally that $\sum_{i = 1}^k \beta_i \le 4k$. Then either $A$ enjoys one of the density increments in \cref{density-increment-calc}, or else
\[ E_{t,k,\vecbeta} \ll \Big( \frac{32r}{\WW} \Big)^{r-1}\Big( \frac{C_0}{k - r} \Big)^{k - r} L^{(\beta_2 + \cdots + \beta_{r}) + k-2r + 1}  \alpha^2 X^2,\]
where the implied constant is absolute. Here, one should interpret $\big( \frac{C_0}{k - r} \big)^{k - r}$ as equal to $1$ in the case $r = k$.
\end{lemma}

Before turning to the proof, we establish some notation and basic results concerning restrictions of $A$ to subprogressions. Let $Q \ge 1$ be an integer and suppose that $b \in \{0,1\dots Q-1\}$. We define the set $A_{Q, b}$ by
\begin{equation}\label{apb-def} 1_{A_{Q,b}}(x') := 1_A(b + Q x').\end{equation}
Since $A$ is supported on $[X]$, $A_{Q,b}$ is supported on $0 \le x' \le \lfloor X/Q\rfloor$. We will use the following identity involving the Fourier coefficients of $A_{Q,b}$, which is essentially the same as \cref{top-p-19}:
\begin{equation}\label{res-fourier} \sum_{\lambda =0}^{Q-1} \Big| \wh{1_A} \Big(\frac{\lambda}{Q} + \theta\Big) \Big|^2 = Q \sum_{b = 0}^{Q-1} \Big| \wh{1_{A_{Q,b}}}(Q\theta)\Big|^2.\end{equation}
To verify the identity, we start with the RHS. We have
\begin{align*}
\wh{1_{A_{Q,b}}}(Q\theta ) & = \sum_{x'} 1_A(b + Qx') e(-Q\theta  x') = \sum_x 1_A(x) 1_{x \equiv b \mdsub{Q}} e \Big({-Q\theta \Big(\frac{x - b}{Q}\Big)} \Big) \\ & = \sum_x 1_A(x) \frac{1}{Q} \sum_{\lambda =0}^{Q-1} e \Big({- \frac{\lambda(x-b)}{Q} }\Big) e \big({-\theta (x - b)} \big) = \frac{1}{Q}e\big(\theta b\big) \sum_{\lambda =0}^{Q-1} \wh{1_A}\Big(\frac{\lambda}{Q} + \theta\Big) e \Big(\frac{\lambda b}{Q}\Big).\end{align*}

Therefore
\begin{align*}
Q\sum_{b  = 0}^{Q-1} \Big| \wh{1_{A_{Q,b}}}(Q \theta) \Big|^2 & = \frac{1}{Q} \sum_{\lambda = 0}^{Q-1}\sum_{\lambda' = 0}^{Q-1} \wh{1_A}\Big(\frac{\lambda}{Q} + \theta\Big)\overline{\wh{1_A}\Big(\frac{\lambda'}{Q} + \theta\Big)} \sum_{b = 0}^{Q-1} e \Big(\frac{(\lambda - \lambda')b}{Q}\Big)  ,  
\end{align*}
which is equal to the LHS of \cref{res-fourier} by orthogonality of additive characters $\md{Q}$.

\begin{proof}[Proof of \cref{n-lem73}]
The key idea is to randomly `sparsify' some of the sets of possible prime factors of elements of $\mathscr{Q}_{t,k,\vecbeta}$. For $r \ge 2$ we choose some random sets of primes $\mathscr{P}_2,\dots, \mathscr{P}_r$ as follows (if $r = 1$ we choose no sets). Set $\delta := (\sum_{i = 2}^r \beta_i)^{-1}$. Independently for each $i = 2,\dots, r$, let $\mathscr{P}_i$ be a random subset of the primes in $[L^{\beta_i}, 2 L^{\beta_i})$ of size $M_i := \lfloor \WW \delta L/\log L\rfloor$. Note that, since $\sum_{i = 2}^k \beta_i \le 4k$ and $\beta_2 \le \cdots \le \beta_k$, we have $\sum_{i = 2}^r \beta_i \le 4k(r-1)/(k-1) \le 8(r-1)$, since $k \ge 2$ in this case; therefore $\delta \ge 1/8r \ge 1/8k$. By \cref{global-conds} it follows that $\delta \ge \frac{\log L }{16 L}$ and so (assuming $W > 16$) we have $M_i \ge W\delta L/2\log L$. In particular, writing $\pi'(x)$ for the number of primes strictly less than $x$,  
\begin{equation}\label{p-prob} \mb{P}(p \in \mathscr{P}_i) = \frac{M_i}{\pi'(2L^{\beta_i}) - \pi'(L^{\beta_i})} \ge \frac{W\delta }{4}L^{1 - \beta_i} \ge \frac{W }{32r}L^{1 - \beta_i},\end{equation} where in the first inequality we used that $\beta_i \ge 1$. Consider the randomly sparsified sum 
\begin{equation}\label{rss} E_{t,k,\vecbeta}(\mathscr{P}_{2},\dots, \mathscr{P}_r) := \sum_{\substack{q \in \mathscr{Q}_{t,k, \vecbeta} \\ p_i(q') \in \mathscr{P}_i \\ i = 2,\dots,r}} \sum_{\substack{1 \le  a < q \\ (a,q) = 1}} \Big|\wh{\mathbf{1}_A}\Big(\frac{a}{q} + \xi\Big)\Big|^2.\end{equation} (This should be compared with \cref{e-def}, which is the same but without the restriction $p_i(q') \in \mathscr{P}_i$.)
By \cref{p-prob} we have
\[ \mb{E} E_{t,k,\vecbeta}(\mathscr{P}_{2},\dots, \mathscr{P}_r) \ge \Big(\frac{\WW}{32r}\Big)^{r-1} L^{r-1 -(\beta_{2} +  \cdots + \beta_r)} E_{t,k,\vecbeta}.\]
Choose a particular instance of the sets $\mathscr{P}_i$ for which 
\begin{equation}\label{instance}  E_{t,k,\vecbeta}(\mathscr{P}_{2},\dots, \mathscr{P}_r) \ge \Big(\frac{W}{32r}\Big)^{r-1} L^{r - 1-(\beta_{2} +  \cdots + \beta_r)} E_{t,k,\vecbeta}.\end{equation}
Set 
\[ Q := t^2 \Big( \prod_{p \le \WW L} p^2 \Big) \cdot \prod_{i = 2}^r \prod_{p \in \mathscr{P}_i} p ^2.\] Recalling the definitions of $\delta$ and $M_i$, and using \cref{global-conds}, we have the bound
\begin{equation}\label{big-P-bd} Q \le t^2 8^{WL} \prod_{i = 2}^r (2L^{\beta_i})^{2M_i} \le t^2 e^{8WL} \ll \alpha^{-4-8W}.\end{equation}
In particular, 
\begin{equation}\label{P-crude} Q < X^{1/10}\end{equation} (say) unless (a stronger bound than) \cref{density-increment-calc} (1) holds, so we may assume this from now on. Now if $q \in \mathscr{Q}_{t,k, p_1,\dots, p_k}$ then $q = t q' = t d p^{\nu_2}_2 \cdots p^{\nu_k}_k$, where $d \mid \prod_{p \le WL} p^2$, $p_i \sim L^{\beta_i}$ and $\nu_i \in \{1,2\}$. If, additionally, $p_i \in \mathscr{P}_i$ for $i = 2,\dots, r$ then $q \mid Q p_{r+1}^2 \cdots p_k^2$, and hence each fraction $\frac{a}{q}$ with $1 \le a < q$ and $(a,q) = 1$ may be written as 
\[ \frac{a}{q} = \frac{\lambda}{Q} +  \sum_{i = r+1}^k \frac{a_i}{p_i^2}\] for some $\lambda$, $1 \le  \lambda < Q$, and $a_i$, $1 \le a_i < p_i^2$. Therefore, from the definition \cref{rss}, we have
\[ E_{t, k,\vecbeta}(\mathscr{P}_{2},\dots, \mathscr{P}_r)  \le \sum_{\substack{p_i \sim L^{\beta_i} \\ i = r_1,\dots, k }}
 \sum_{\substack{1 \le \lambda < Q \\ 1 \le a_i < p_i^2 \\ i = r+1,\dots, k}} \Big| \wh{1_A} \Big( \frac{\lambda}{Q} + \sum_{i = r+1}^{k} \frac{a_i}{p^2_i} + \xi \Big) \Big|^2 .\]
With an application of \cref{thm:level-d-int} in mind, take $\mc{Q}$ to consist of $p^2$ for all primes $p$ with $WL < p \le C_1 \alpha^{-2}$ other than those dividing $t$ or in one or more of the sets $\mathscr{P}_2,\cdots,\mathscr{P}_r$. 
By \cref{res-fourier} (and extending the sum to include $\lambda = 0$, and relaxing the conditions $p_i \sim L^{\beta_i}$ to $p^2_i \in \mc{Q}$) it follows that 
\begin{equation} E_{t, k,\vecbeta}(\mathscr{P}_{2},\dots, \mathscr{P}_r) \le 
Q \sum_{b = 0}^{Q-1} \sum_{p_{r+1}^2 < \dots < p_k^2 \in \mc{Q} }\sum_{\substack{ 1 \le a_i < p_i^2\\ i = r+1,\dots, k}} \Big| \wh{1_{A_{Q, b}}} \Big( Q \Big( \sum_{i = r+1}^k \frac{a_i}{p^2_i} + \xi \Big)\Big) \Big|^2. \label{8281} \end{equation}

Now it is time for the application of \cref{thm:level-d-int}. Take $\mc{Q}$ as above. For each fixed $b$ in \cref{8281}, we seek to bound the remaining sum (that is, over the $p_i^2$ and $a_i$) using \cref{main41}. If $r \le k - 1$, one may observe that this sum is exactly of the form of the LHS of \cref{main41}, with $d = k - r$ and $f : [0, \lfloor X/Q\rfloor ] \rightarrow \C$ defined by $f(x) := 1_{A_{Q, b}}(x) e(-Q \xi x)$. Here, we observe that $Q$ is coprime to each $p_i$, $i = r+1,\dots, k$ by construction, so the $Qa_i$ range over nonzero residues modulo $p_i^2$ as the $a_i$ do. Thus, if \cref{main41} holds in this context (for all $b$) we conclude that 
\begin{equation}\label{e-bd-2}  E_{t, k,\vecbeta}(\mathscr{P}_{2},\dots, \mathscr{P}_r) \ll Q^2 \cdot \alpha^2 \Big( \frac{X}{Q} \Big)^2 \Big( \frac{C_0 L}{k-r} \Big)^{k - r}.\end{equation} (The $\ll$ easily absorbs the small error arising from the two endpoints of $[0, \lfloor X/Q\rfloor]$, which (depending on $b$) might not be in the support of $A_{Q,b}$, as well as the related issue that \cref{thm:level-d-int} is stated for intervals starting at $1$.) Combining with \cref{instance} gives the upper bound stated in the lemma.

It remains to analyse what happens if we do not have \cref{main41} for some $b \in \{0,1\cdots, Q-1\}$.  We first check the conditions in \cref{thm:level-d-int}. The condition $\max_{q \in \mc{Q}} q \le X^{1/32L}$ (with $X$ replaced by $X/Q > X^{1/2}$) is implied by $C_1 \alpha^{-2} \le X^{1/64 L}$, which holds unless \cref{density-increment-calc} (1) holds. The value $d = k - r$ satisfies the required condition $1 \le d \le 2^{-7} L$ in \cref{thm:level-d-int} by \cref{global-conds} (and since $L$ is sufficiently large). By \cref{thm:level-d-int}, the only other way that we could fail to have \cref{main41} (and hence, as already shown, the bound claimed in the present lemma) is if there is some set $S \subseteq \mc{Q}$, $1 \le |S| \le 2L$, and some $u \in \Z$ such that the average of $|f(x)|$ on the progression $\{ x \in [0, X/Q] : x \equiv u \md{\prod_{q \in S} q}\}$ is greater than $2^{|S|} \alpha$.

Recalling definitions of $f$ and of $A_{Q,b}$ (see \cref{apb-def}), this means that $A$ has density at least $2^{|S|}\alpha$ on the subprogression $\{ x \in [X] : x \equiv b + Qu \md{Q \prod_{q \in S} q}\}$. Note that the modulus here is a square. By \cref{big-P-bd}, this gives a density increment of type \cref{density-increment-calc} (2), taking $m = |S|$ and any $C > 6 + 8W$ there. 

In the case $r = k$, a minor modification is required since \cref{thm:level-d-int} does not strictly apply as written. However, one can instead simply use the fact that $\big| \wh{1_{A_{Q,b}}}(Q \xi) \big| \le 2 \alpha$ unless $A_{Q,b}$ has density at least $2 \alpha$ (on a progression of length $\approx X/Q$) and then argue as above, with an extra factor $2$ on the RHS of \cref{e-bd-2} which may of course be absorbed into the $\ll$ notation in the statement of the lemma. This concludes the proof.
\end{proof}

\section{Completing the argument}\label{sec8}

In this section, we complete the task of deducing a density increment as in \cref{density-increment-calc}, assuming \cref{just-before-thm12}, and assuming also that $\alpha \le c$ for sufficiently small absolute constant $c$. Our main tool will be \cref{n-lem73}. We retain the notation and setup from \cref{sec8-1}. Define

\begin{equation}\label{stkbeta-def} S_{t,k, \vecbeta} := \sum_{q \in \mathscr{Q}_{t,k, \vecbeta}} q^{-1/2} \sum_{\substack{ 1 \le a < q \\ (a,q) = 1}} \Big| \wh{1_A}\Big(\frac{a}{q} + \xi\Big) \Big|^2. \end{equation}
We first note that \cref{n-lem73} implies a bound on $S_{t,k,\vecbeta}$ by observing that if $q \in \mathscr{Q}_{t,r, \vecbeta}$ then $q \ge tL^{\beta_1 + \cdots + \beta_r}$
since $t$, $d(q')$ and the $p_i(q')$, $i = 2,\dots,k$ are coprime and divide $q$, and $d(q') \sim L^{\beta_1}$, $p_i(q') \sim L^{\beta_i}$. 
Therefore, for all $r$ with $1 \le r \le k$ we have
\begin{align}\nonumber S_{t,k,\vecbeta} & \ll t^{-1/2}L^{-(\beta_1 + \cdots + \beta_r)/2}E_{t,k,\vecbeta} \\ & \ll t^{-1/2}\big( \frac{32r}{\WW} \big)^{r-1}\big( \frac{C_0}{k - r} \big)^{k - r} L^{(\beta_{2} + \cdots + \beta_r) + k - 2r + 1 - (\beta_1 + \cdots + \beta_k)/2}  \alpha^2 X^2,\label{s-bd-3}\end{align} in each case provided that we do not already have a density increment as in \cref{density-increment-calc}, and in the case $r \ge 2$ provided that we additionally assume $\sum_{i = 1}^k \beta_i \le 4k$. We recall that by convention $\big( \frac{C_0}{k - r} \big)^{k - r} = 1$ when $r = k$.

Write $\delta_0$ for the implied constant in \cref{just-before-thm12}. Thus this result states that
\[ \sum_{2 \le q \le C_1\alpha^{-2}} q^{-1/2}\sum_{\substack{1 \le a < q \\ (a,q) = 1}}\Big|\wh{\mathbf{1}_A}\Big(\frac{a}{q} + \xi\Big)\Big|^2  \ge \delta_0\alpha^2  X^2.\]
Since the sets $\mathscr{Q}_{t,k,\vecbeta}$ partition the set of $q$s, it follows that 
\begin{equation}\label{7972}  \sum_{t \; \cubefull} \sum_{k =1}^{\infty} \sum_{\vecbeta \in \mc{D}^k} S_{t,k,\vecbeta}
 \ge \delta_0\alpha^2  X^2,\end{equation} where here $\mc{D} = (\log L)^{-1} \Z_{\ge 0}$.

\subsection{Discarding large parameters}\label{discard-large-sec}

Before turning to the main analysis, we first dispense with the case $\sum_{i = 1}^k \beta_i > 4k$; we will then be able to access \cref{s-bd-3} for all $r$.

\begin{lemma}\label{large-param-discard}
Suppose that $W$ is sufficiently large. Then either there is a density increment as in \cref{density-increment-calc}, or else the contribution to \cref{7972} from parameter sets with $\sum_{i = 1}^k \beta_i \ge 4k$ is at most $\frac{1}{2}\delta_0 \alpha^2 X^2$.   
\end{lemma}
\begin{proof}
We will use \cref{s-bd-3} with $r = 1$, noting carefully that there is no assumption on $\sum_{i=1}^k \beta_i$ in this case. For fixed $m \in \N$ we consider the contribution from $\sum_{i = 1}^{k} \beta_i = (\log L)^{-1} m$; at the end we will sum over $m \ge 4k \log L$. Recall that the $\beta_i$ are constrained to lie in $(\log L)^{-1} \Z_{\ge 0}$. Therefore the number of tuples $(\beta_1 , \cdots , \beta_k)$ with $\sum_{i = 1}^k \beta_i = (\log L)^{-1} m$ is $\le \binom{m + k - 1}{k - 1} \le 4(em/k)^k$. (Here we are counting without the extra constraints \cref{beta-ordering}; there turns out to be no need to include them here.) By \cref{s-bd-3} with $r = 1$ we 
have
\[  S_{t,k,\vecbeta}  \ll t^{-1/2}\Big( \frac{C_0}{k - 1} \Big)^{k - 1} L^{k -  1}e^{-m/2}  \alpha^2 X^2 \] for each of them, or else a density increment as in \cref{density-increment-calc}. The contribution to \cref{7972} from $\sum_{i = 1}^k \beta_i > 4k$ (summed over all $t$ and $k$) is therefore
\[ \ll \alpha^2 X^2\sum_{t \; \cubefull} t^{-1/2}\sum_{k = 1}^{\infty} \Big( \frac{C_0}{k - 1} \Big)^{k - 1} L^{k - 1}\sum_{m \ge 4k \log L} \Big(\frac{em}{k}\Big)^k  e^{-m /2} .\]
We can ignore the sum over $t$ by \cref{38-bd}. For $m \ge 4k \log L$ we have $L^{k-1} e^{-m/2} \le L^{-1} e^{-m/4}$ and so we may in turn bound this by
\[ \ll \alpha^2 X^2 L^{-1} \sum_{k = 1}^{\infty} (eC_0)^k k^{1-2k} \sum_{m =1}^{\infty} m^k e^{-m/4}.\] 
We have
\[ \sum_{m = 1}^{\infty} m^k e^{-m/4} \le \max_m \big(m^k e^{-m/8}\big) \cdot \sum_{m = 1}^{\infty} e^{-m/8} \le (1 - e^{-1/8})^{-1} (8k/e)^k .\]
Thus we see that the contribution to \cref{7972} from $\sum_{i = 1}^k \beta_i > 4k$ is
\[ \ll \alpha^2 X^2 L^{-1} \sum_{k = 1}^{\infty} (8C_0)^k k^{1-k} \ll \alpha^2 X^2 L^{-1},\] which is at most $\frac{1}{2}\delta_0 \alpha^2 X^2$ since $L$ is sufficiently large.
\end{proof}

\subsection{The contribution of $k \ge 2$} \label{sec8.2}

In this section we show that the contribution of the terms with $k \ge 2$ to \cref{7972} is negligible, provided that $W$ is sufficiently large. For this analysis it seems clearer to stop tracking constants to a certain extent, thus for instance we denote by $O(1)$ an absolute constant which can change from line to line. 

In following the argument, the reader may find it helpful to note that when $\vecbeta = (\beta_1,\dots, \beta_k) = (0,2,\dots,2)$ (and $t = 1$, say) none of the inequalities \cref{s-bd-3} for $1 \le r \le k$ give much of a saving over $\alpha^2 X^2$. It may also be checked that this is the only parameter choice for which this is so. The argument is in a sense an analysis of what happens `near' these critical vectors $\vecbeta$; the reader may find this helpful in motivating the definition \cref{bjk-def}, for instance.

We isolate the following lemma from the proof.

\begin{lemma}\label{partition-lem}
The number $V(s,n)$ of tuples $(n_1, \dots, n_s) \in \Z_{\ge 0}^s$ with $n_1 + \cdots + n_s = n$ and $n_1 \le n_2 \le \cdots \le n_s$ is $\le (n + s^2)^s/(s!)^2$.
\end{lemma}
\begin{proof}
Define new variables $u_i$ by $n_i = u_1 + \cdots + u_i$; then the $u_i$ are non-negative integers satisfying $u_1 + 2u_2 + \cdots + su_s = n$. If $(n_1,\dots, n_s)$ is a tuple counted by $V(s,n)$ then the cube $(u_1,\dots, u_s) + [0,1]^s$ is contained in $\{ (y_1,\dots, y_s) \in \R^s_{\ge 0} : y_1 + 2y_2 +\cdots + sy_s \le n + s^2\}$, which has volume $(n + s^2)^s/(s!)^2$, and these cubes are disjoint for different $(n_1,\dots, n_s)$. 
\end{proof}

We now turn to the main argument. By the results of \cref{discard-large-sec}, the bound \cref{7972} may be replaced by
\begin{equation}\label{7973}  \sum_{t \; \cubefull} \sum_{k =1}^{\infty} \sum_{\substack{\vecbeta \in \mc{D}^k  \\  \sum_{i = 1}^k \beta_i \le 4k }} S_{t,k,\vecbeta}
 \ge \frac{1}{2}\delta_0\alpha^2  X^2,\end{equation} 
where (recall) the definition of $S_{t,k,\vecbeta}$ is \cref{stkbeta-def}.

In this range of parameters, we have access to \cref{s-bd-3} for all $r$ with $1 \le r \le k$. Our plan is to use this to show that the contribution to \cref{7973} from $k \ge 2$ is at most $\frac{1}{4}\delta_0 \alpha^2 X^2$. Then, at the end, we will analyse the contribution from $k = 1$ separately using \cref{dens-increment-lem}.

The handling of the case $k \ge 2$ using \cref{s-bd-3} is a pure estimation problem, but it is slightly delicate. Let us recall that we may assume $\beta_1 \ge 0$ and $1 \le \beta_2 \le \beta_3 \le \cdots \le \beta_k$ (or else $S_{t,k, \vecbeta} = 0$). We further fibre the tuples $\vecbeta = (\beta_1,\cdots, \beta_k)$ according to the value of $j \in \{1,\dots, k\}$, which we define to be the smallest index $j \ge 1$ such that $\beta_{j+1} \ge 2$ (or $j = k$ if no such index exists). For notational brevity set (recalling the definition \cref{beta-ordering} of $B_k$)
\begin{equation}\label{bjk-def}  B_{j,k} := \{ \vecbeta \in B_k :  1 \le \beta_2 \le \dots \le \beta_{j} < 2 ,\;  2 \le \beta_{j+1} \le \dots \le \beta_k\},\end{equation}
where one should ignore the last constraint when $j = k$, and the first one when $j = 1$.
Thus the contribution to \cref{7973} from $k \ge 2$ becomes
\begin{equation}\label{7973-bis}  \sum_{t \; \cubefull} \sum_{k \ge 2}\sum_{j = 1}^k \sum_{\substack{\vecbeta \in B_{j,k} \\ \sum_{i = 1}^k \beta_i \le 4k}}  S_{t,k,\vecbeta}
 .\end{equation} 

To estimate the quantities $S_{t,k,\vecbeta}$, we take the geometric mean of \cref{s-bd-3} at the values $r = 1, k$, and also at the values $r = 1, j, k$. This gives, respectively,
\begin{equation}\label{product-s0}
S_{t,k, \vecbeta} \ll t^{-1/2} \Big( \frac{O(1)}{W}\Big)^{(k-1)/2} \alpha^2 X^2
\end{equation}
and 
\begin{equation}\label{product-s} S_{t,k, \vecbeta} \ll t^{-1/2}\Big(\frac{O(k)}{W}\Big)^{(k-1)/3}  L^{\frac{1}{3} (-\beta_1 + \frac{1}{2}(\beta_2 + \cdots + \beta_j - \beta_{j+1} - \cdots - \beta_k) + k - 2j + 1)}.\end{equation} (In the first estimate we could include an extra factor of $L^{-\frac{1}{2}\beta_1}$, but we do not need it.) Note the exponent of $L$ in \cref{product-s} is always non-positive due to the constraints \cref{bjk-def} on the $\beta_j$. 

For $m \in \Z_{\ge 0}$, we introduce the further refinement
\[ B_{j,k,m} := \Big\{ \vecbeta \in B_{j,k}:  \beta_1 < \frac{m}{\log L}, -\beta_2  - \cdots - \beta_j + \beta_{j+1} + \cdots + \beta_k < 2(k - 2j + 1)  + \frac{m}{\log L}\Big\}. \] Note that $B_{j,k,0} = \emptyset$.  From \cref{product-s0,product-s} we see that 
\[ \max_{\vecbeta \notin B_{j,k,m}} S_{t,k; \vecbeta} \ll  t^{-1/2}  W^{-(k-1)/3}\alpha^2 X^2 \min\Big( 
O(k)^{k} e^{-m/6}, O(1)^{k}\Big).\] By considering, for each $\vecbeta$, the smallest $m$ for which $\vecbeta \in B_{j,k,m}$, since $\vecbeta \notin B_{j, k, m-1}$ it follows that 
\begin{align} \nonumber \sum_{\substack{\vecbeta \in B_{j,k} \\ \sum_{i = 1}^k \beta_i \le 4k}} S_{t,k; \vecbeta} & \ll  \sum_{m = 1}^{\infty} |B_{j,k,m}| \max_{\vecbeta \notin B_{j,k,m-1}} S_{t,k; \vecbeta} \\ & \ll t^{-1/2}W^{-(k-1)/3} \alpha^2 X^2 \sum_{m = 1}^{\infty} |B_{j,k,m}| \min\Big( O(k)^{k} e^{-m/6},  O(1)^{k}\Big).\label{s-bd-22}
\end{align}
To proceed further, we need an upper bound for $|B_{j,k,m}|$. We claim that 
\begin{equation}\label{bjkm- upper}  |B_{j,k,m}| \ll \left\{ \begin{array}{ll}    O(1)^k & \mbox{uniformly for $m \le k^2$} \\ O(1)^k k^{-2k} m^{k+2} & \mbox{uniformly for $m \ge k^2$.} \end{array}\right.\end{equation}
To prove the claim, first make the substitution $x_1 := \beta_1 \log L$, $x_i := |\beta_i - 2|\log L$ for $2 \le i \le k$. Then $(x_1,\dots, x_k) \in \Z_{\ge 0}^k$ and $B_{j,k,m}$ is defined by the conditions
\begin{equation}\label{x1} 0 \le x_1 < m, \;\;  x_2  + \cdots + x_k <  m,\;\; 0 \le x_{j+1} \le \cdots \le x_k, \;\;  0 < x_j \le \cdots \le x_2.\end{equation}
Writing $\ell = x_2 + \cdots + x_j$ and $\ell' = x_{j+1}+\cdots + x_k$, it follows from \cref{partition-lem} (which also gives the definition of $V(\; , \;)$) that we have
\begin{align*} |B_{j,k,m}| & \le m \sum_{\ell +\ell' \le m} V(j-1, \ell) V(k - j, \ell') \\ & \ll m^3 \frac{(m + k^2)^{k-1}}{(j - 1)!^2 (k - j)!^2} \le \frac{m^3 4^{k-1} (m + k^2)^{k - 1}}{(k-1)!^2},\end{align*} where here we used the binomial coefficient bound $\binom{k-1}{j-1} \le 2^{k-1}$. The claim \cref{bjkm- upper} now follows using the inequality $x! \ge (x/e)^x$ with $x = k-1$, and considering the two cases separately. 

Returning to \cref{s-bd-22}, we split the sum into $m \le k^2$ and $m > k^2$. Using the second argument of the $\min(\; , \;)$ in \cref{s-bd-22}, together with the first bound in \cref{bjkm- upper}, the sum over the first range is $\ll t^{-1/2} W^{-(k-1)/3} O(1)^{k}\alpha^2 X^2 $. (Here, of course, the $k^2$ terms in the sum can easily be absorbed into the $O(1)^k$.) To estimate the sum over $m > k^2$ we use the first argument of the $\min(\; , \;)$ and the second inequality in \cref{bjkm- upper}, obtaining a bound of
\[ \ll t^{-1/2} W^{-(k-1)/3} \alpha^2 X^2 O(k)^{k} k^{-2k} \sum_{m \ge k^2} m^{k+2} e^{-m/6} \ll t^{-1/2} W^{-(k-1)/3} \alpha^2 X^2 .\]

Summarising, we have shown that 
\[ \sum_{\substack{\vecbeta \in B_{j,k} \\ \sum_{i = 1}^k \beta_i \le 4k}} S_{t,k,\vecbeta} \ll t^{-1/2} W^{-(k-1)/3} O(1)^k \alpha^2 X^2. \]
Summing this from $j = 1$ to $k$, then over $k \ge 2$, and finally over $t$ cube-full (using \cref{38-bd}), we see that if $W$ is sufficiently large then the quantity \cref{7973-bis}, that is to say the contribution to \cref{7973} from $k \ge 2$, is $\le \frac{1}{4}\delta_0 \alpha^2 X^2$. This means that the contribution to \cref{7973} from $k = 1$ is $\ge \frac{1}{4}\delta_0 \alpha^2 X^2$.

\subsection{Handling $k = 1$} By the analysis of the last subsection, we may now discard all terms with $k \ge 2$ from \cref{7973}, at the expense of replacing $\frac{1}{2}\delta_0$ by $\frac{1}{4}\delta_0$, provided that $W$ is sufficiently large. Fix some $W = O(1)$ for which this holds true.

That is, we may assume that

\begin{equation}\label{7975}  \sum_{t \; \cubefull}  \sum_{\substack{\beta_1 \in \mc{D}  \\  \beta_1 \le 4 }} S_{t,1,\beta_1}
 \ge \frac{1}{4}\delta_0\alpha^2  X^2.\end{equation} 

Unpacking the definition \cref{stkbeta-def} of $S_{t,1,\beta_1}$, it follows that 

 \[  \sum_{t \; \cubefull}  \sum_{\beta_1 \in \mc{D} }\sum_{q \in \mathscr{Q}_{t,1, \beta_1}} q^{-1/2} \sum_{\substack{ 1 \le a < q \\ (a,q) = 1}} \Big| \wh{1_A}\Big(\frac{a}{q} + \xi\Big) \Big|^2 \ge \frac{1}{4}\delta_0\alpha^2  X^2. \]
 Recalling the definitions in \cref{sec8-1}, we see that the sum ranges over all $q$, $2 \le q \le C_1 \alpha^{-2}$ for which the cube-free part $q'$ divides $\prod_{p \le WL} p^2$. That is, 
\[  \sum_{t \; \cubefull} \sum_{\substack{2 \le q \le C_1 \alpha^{-2} \\ q_* = t \\ q' \mid \prod_{p \le WL} p^2}}  q^{-1/2} \sum_{\substack{ 1 \le a < q \\ (a,q) = 1}} \Big| \wh{1_A}\Big(\frac{a}{q} + \xi\Big) \Big|^2 \ge \frac{1}{4}\delta_0\alpha^2  X^2. \]
By \cref{38-bd}, there is some $t \le C_1 \alpha^{-2}$ such that 
\[  \sum_{\substack{2 \le q \le C_1 \alpha^{-2} \\ q_* = t \\ q' \mid \prod_{p \le WL} p^2}} (q')^{-1/2} \sum_{\substack{ 1 \le a < q \\ (a,q) = 1}} \Big| \wh{1_A}\Big(\frac{a}{q} + \xi\Big) \Big|^2 \gg \alpha^2  X^2. \]
This certainly implies that 
\[  \sum_{\substack{q \mid t\prod_{p \le WL} p^2 \\ 2 \le q \le C_1 \alpha^{-2}}} \sum_{\substack{ 1 \le a< q \\ (a,q) = 1}} \Big| \wh{1_A}\Big(\frac{a}{q} + \xi\Big) \Big|^2 \gg \alpha^2  X^2. \] As in \cref{m-prime-balanced} (and since $\alpha > X^{-1/100}$) it follows that 
\[  \sum_{\substack{q \mid t\prod_{p \le WL} p^2 \\  q \le C_1 \alpha^{-2}}} \sum_{\substack{1 \le a < q \\ (a,q) = 1}} \Big| \wh{f_A}\Big(\frac{a}{q} + \xi\Big) \Big|^2 \gg \alpha^2  X^2, \] where $f_A = 1_A - \alpha 1_{[X]}$ is the balanced function of $A$.  Putting all the rationals here over a common denominator of $Q := t \prod_{p \le WL} p^2$, we therefore have

\[  \sum_{b = 0}^{Q-1} \Big| \wh{f_A}\Big(\frac{b}{Q} + \xi\Big) \Big|^2 \gg \alpha^2  X^2. \]  As in \cref{big-P-bd} (and since $W$ is now fixed), we have $Q \ll \alpha^{-O(1)}$.
This puts us in the situation covered by \cref{dens-increment-lem}. In that lemma as applied to the present situation, we have $T \le C_1 L^2$ by the definition \cref{major-arc} of $\tau$ and the fact that $|\xi| \le \tau$ (see just before \cref{just-before-thm12} for discussion), and $\eta \gg 1$. The lemma therefore gives the existence of a progression $P \subset [X]$ with common difference $Q^2$ and length $\gg Q^{-2} T^{-1} X \gg \alpha^{O(1)} X$ on which the density of $A$ is at least $(1 + c) \alpha$ for some absolute $c > 0$. This is a density increment of type \cref{density-increment-calc} (3). 

Starting from \cref{just-before-thm12}, we have now obtained a density increment in all cases of the analysis, and this completes the proof of \cref{density-increment-calc} and hence of \cref{thm:main}.

\section{Limitations of the density increment approach}\label{sec7}

We end with discussion of an example which limits the density increment method to bounds of quasi--polynomial shape. 

\begin{proposition}\label{main-sec7}
Let $X$ be a large integer parameter and let $\eta > 0$. Suppose that \[ \exp(-(\log X)^{3/4 - \eta}) \le \alpha \le \alpha_0(\eta),\] where $\alpha_0(\eta)$ is a sufficiently small absolute constant. Then there is a function $f : [X] \rightarrow [0,1]$ with the following properties: \begin{enumerate}\item[\textup{(1)}]  $\mbm{E}_{x \in [X]} f(x)= \alpha$;
\item[\textup{(2)}]   the sum of $f(x) f(y)$ over all $x, y \in [X]$ with $x - y$ a square is at most $\frac{1}{100} \alpha^2 X^{3/2}$;
\item[\textup{(3)}]  the length $X'$ of the longest subprogression of $[X]$ on which $f$ has density $\ge 2\alpha$ satisfies $X' < X \exp (-\log (1/\alpha)^{1/3})$.\end{enumerate}
\end{proposition}
\begin{remarks}
By a random sampling argument one can show the existence of a set $A \subset [X]$ such that the characteristic function $\mathbf{1}_A$ has similar properties to the function $f$; we omit the detailed statement and proof.

We interpret this result as saying that the density increment method cannot give a bound on $s(X)$ superior to $s(X) \ll X \exp(-c(\log X)^{3/4})$ at least in anything like the form in the present paper (and in previous works). Let us make a couple of points to support this view.

First of all we make the comment that from the perspective of Fourier analysis, the statement that a set $A$ of density $\alpha$ has $\le \frac{1}{100} \alpha^2 X^{3/2}$ square differences is just as good as the assumption that $A$ contains no square differences at all: to be more specific, the analysis starting at \cref{no-square-diff} would work just as well with this assumption. 

Our second point is that if one takes the statement `there is a subprogression of length $X' = X \exp (-\log (1/\alpha)^{1/3})$ on which the density of $A$ is at least $2\alpha$' and then uses this in a density increment argument (or equivalently an inductive argument along the lines of \cref{density-increment-calc}) one ends up proving a bound $s(X) < X \exp (-c(\log X)^{3/4})$ (and no better). We leave the details of this check to the reader.

\end{remarks}

\begin{proof}
To begin the construction, let $p$ be a prime with $p \equiv 1 \md{4}$. For a parameter $\eps \in [0,1]$ to be specified shortly, and define $\psi_p : \Z/p\Z \rightarrow [0,1]$ by
\[ \psi_p(x) := (1 + \eps)^{-1} \big(1 - \eps \cos (2\pi sx/p)\big),\] where $s$ is a quadratic nonresidue modulo $p$. We note that 
\[  \wh{\psi_p}(r) = \mbm{E}_{x \in \Z/p\Z} \psi_p(x) e\Big({-\frac{rx}{p}}\Big) = (1 + \eps)^{-1} \Big(\mathbf{1}_{r = 0} - \frac{\eps}{2} \mathbf{1}_{r = s} - \frac{\eps}{2} \mathbf{1}_{r = -s}\Big)\] for all $r \in \Z/p\Z$.
Note also that from standard Gauss sum estimates (see for instance \cite[Equation (3.29)]{IK-book}) we have
\[ \wh{\mathbf{1}_{\Box}}(s) = \frac{1}{2p} \Big( 1 + \legendre{s}{p} \sqrt{p}\Big) = \frac{1 - \sqrt{p}}{2p}, \] where $\Box$ denotes the set of squares modulo $p$ (including zero). We therefore compute that if
\begin{equation}\label{eps-p-cond} \eps \ge 4 p^{-1/4} \end{equation} then we have
\begin{align}\nonumber 
\mb{E}_{x, y \in \Z/p\Z} & \psi_p(x) \psi_p(y) \mathbf{1}_{\Box}(x - y)   = \sum_r \big| \wh{\psi_p}(r) \big|^2 \wh{1_{\Box}}(r)   = \big| \wh{\psi_p}(0) \big|^2\Big( \wh{1_{\Box}}(0) + \frac{\eps^2}{2} \wh{1_{\Box}}(s)\Big) \\ &   = \big( \mbm{E} \psi_p \big)^2 \Big( \frac{p+1}{2p} + \frac{\eps^2}{2} \Big(  \frac{1 - \sqrt{p}}{2p}  \Big)\Big) \le \big( \mbm{E} \psi_p \big)^2\frac{1}{2}\Big(1 - \frac{\eps^2}{8\sqrt{p}}\Big),\label{p-squares-est}
\end{align}
where we used \cref{eps-p-cond} in the last step.
The next step is to take a product of such functions over a set of primes $p$. Let $C$ be some large absolute constant to be specified later. Let $T \ge C^4$ be a parameter such that if we set
\begin{equation}\label{eps-def} \eps := C(\log T)^{1/2} T^{-1/4}\qquad \mbox{and} \qquad M := \lfloor T/4\log T\rfloor\end{equation}
then
\[ \alpha \le (1 +  \eps)^{-M} \le 2\alpha. \]
It is easy to see that (if $\alpha$ is small enough) there will be such a parameter and that we will have 
\begin{equation}\label{alpha-T-eq} T^{3/4} (\log T)^{-1/2} \asymp \log (1/\alpha).\end{equation} Let $\mc{P}$ be a set of $M$ primes in $[T, 2T]$, all $1 \md{4}$. Such a set will exist if $\alpha$ is sufficiently small (and hence $T$ is sufficiently large). Let $N := \prod_{p \in \mc{P}} p$. Taking a product of the functions $\psi_p$ (and applying the Chinese remainder theorem), we obtain a function $\psi_N : \Z/N\Z \rightarrow [0,1]$ such that, using \cref{p-squares-est}, we have
\begin{equation}\label{avg-sqs} \mb{E}_{x, y \in \Z/N\Z} \psi_N(x) \psi_N(y) \mathbf{1}_{\Box}(x - y) \le 2^{-M} (\mb{E} \psi_N)^2 \prod_{p \in \mc{P}} \Big(1 - \frac{C^2\log T}{8(pT)^{1/2}}\Big) \le 2^{-M} (\mb{E} \psi_N)^2  e^{-C^2/64}.\end{equation}
Here, $\Box$ denotes the squares modulo $N$; note also that the condition \cref{eps-p-cond} required for \cref{p-squares-est} to hold is satisfied. Note that 
\begin{equation}\label{psi-N-avg} \mbm{E} \psi_N = (1 + \eps)^{-M} \in [\alpha, 2\alpha].\end{equation}
Let us also note that $N \le 4^{2T}$ so, in view of \cref{alpha-T-eq} and the assumption that $\alpha > \exp (-(\log X)^{3/4 - \eta})$, we have $N < X^{1/10}$ provided that $\alpha$ is sufficiently small.

Finally we come to the definition of the function $f$ itself. We define
\[ f(x) := \left\{ \begin{array}{ll} c\psi_N (x \md{N}) & x \le N \lfloor X/N\rfloor  \\ \alpha &  N \lfloor X/N\rfloor < x \le X. \end{array}\right.\]
where $c \in [\frac{1}{2},1]$ is chosen so that $c \mbm{E} \psi_N = \alpha$ (which is possible by \cref{psi-N-avg}). By construction, we have $\mbm{E}_{[X]} f = \alpha$, which is property (1) in the proposition.

Turning to (2), we abuse notation by writing $\Box$ for the set of squares in $\N$ (retaining the earlier use of the same symbol for the squares in $\Z/N\Z$). We wish to bound the sum 
\[ S := \sum_{x,y\in [X]}f(x)f(y)\mathbf{1}_{\Box}(x - y).\]
The contribution to $S$ from $x$ or $y$ larger than $N \lfloor X/N\rfloor$ is $\ll NX$ (ignoring the constraint $x - y \in \Box$ completely) which is negligible compared to $\alpha^2 X^{3/2}$ by the assumption on $\alpha$ and the bound $N < X^{1/10}$. For the remaining portion of the sum, we have 
\begin{equation} \sum_{1 \le x,y\le N\lfloor X/N\rfloor}f(x)f(y)\mathbf{1}_{\Box}(x - y) = c^2\sum_{a,b\in \Z/N\Z}\psi_N(a)\psi_N(b)\mathbf{1}_{\Box}(a-b)\sum_{\substack{1 \le x, y\le N\lfloor X/N\rfloor \\ x\equiv a\mdsub{N}\\ y\equiv b\mdsub{N}}} \mathbf{1}_{\Box}(x - y)\label{eq75}.
\end{equation}
For fixed $y$, the inner sum over $1 \le x \le N \lfloor X/N\rfloor$ is bounded above by the number of squares in $[X]$ which are congruent to $a - b \md{N}$. Such squares must be $x^2$ with $x \equiv t \md{N}$, where $t$ is one of the square roots of $a - b \md{N}$, of which there are at most $2^M$. Thus for fixed $y$ the inner sum over $x$ in \cref{eq75} is $\ll 2^M (X^{1/2}/N)$. For fixed $b$, there are $\ll X/N$ choices of $y$. Summing over these, it follows that
\begin{align*}
\sum_{1 \le x,y\le N\lfloor X/N\rfloor}f(x) f(y)\mathbf{1}_{\Box}(x - y) & \ll 2^{M} X^{3/2} N^{-2}\sum_{a,b\in \Z/N\Z}\psi_N(a)\psi_N(b) \mathbf{1}_{\Box}(a-b) \\ & = 2^M X^{3/2}\mb{E}_{\Z/N\Z} \psi_N(a)\psi_N(b)\mathbf{1}_{\Box}(a-b).\end{align*}
By \cref{avg-sqs} and \cref{psi-N-avg} it follows that 
\[ 
\sum_{1 \le x,y\le N\lfloor X/N\rfloor} f(x) f(y)\mathbf{1}_{\Box}(x - y) \ll  \alpha^2 X^{3/2}e^{-C^2/64}. \]
Taking $C$ sufficiently large, we see that $S \le \frac{1}{100}\alpha^2 X^{3/2}$, which is what we wanted to prove.

Finally we turn to (3). We first consider the case that $P$ is a maximal subprogression of $[X]$ whose common difference is a product $p_1 \cdots p_m$ of distinct primes from $\mc{P}$.
By inspection of the construction one may see that the average of $\psi_N$ over any congruence class modulo $p_1 \cdots p_m$ is at most $(1 + \eps)^{m - M}$. Thus, by the construction of $f$ we see that the average of $f$ over $P$ is at most $(1 + \eps)^m \alpha$. Now let $P$ be an arbitrary subprogression of $[X]$, of length $X'$. It is contained in a maximal subprogression $\tilde P$ of the above form, and unless $X'$ is already very small, the averages of $f$ on $P$ and on $\tilde P$ are very close; we certainly have that the average of $f$ on $P$ is bounded above by $\frac{3}{2} (1 + \eps)^m \alpha$.

Thus we see that in order to double the average of $f$ from $\alpha$ to $2\alpha$, we must take $(1 + \eps)^m \ge \frac{4}{3}$, which means that 
\[ X' \le X T^{-m} \le X \big( \frac{4}{3}\big)^{-\log T/\log(1 + \eps)} < X \exp (-T^{1/4}) < X \exp (-\log(1/\alpha)^{1/3}),\]
where here we used the definition \cref{eps-def} of $\eps$ and \cref{alpha-T-eq}, and the fact that $T, 1/\alpha$ are sufficiently large.\end{proof}

We remark that a minor modification of this construction shows that the density increment argument cannot give better than quasipolynomial bounds for related (but seemingly easier) questions such as locating a non-zero square in $2A - 2A$.

\appendix

\section{A hypercontractive inequality of Keller, Lifshitz and Marcus}\label{appA}

In this appendix we deduce \cref{thm:input-hypercontractive} from \cite[Corollary 4.7]{KLM23}. As remarked after the statement of the former result, this amounts to checking that the statement of \cite[Corollary 4.7]{KLM23}, which is given for functions $f : \Omega^n \rightarrow \R$, remains valid for functions $f : \Omega_1 \times \cdots \times \Omega_n \rightarrow \C$. All of the nomenclature of \cite{KLM23} adapts from the setting of $\Omega^n$ to a general product $\Omega_1 \times \cdots \times \Omega_n$ with essentially no change. In particular, this is so for the key definition of $(r,\gamma)$-$L_2$-global (see \cite[Definition 4.4]{KLM23}) and the definitions that feed in to it. We use the nomenclature of \cite{KLM23} (extended from $\Omega^n$ to general products) without further comment. 

First we extend \cite[Corollary~4.7]{KLM23} to the case of functions $f : \Omega^n \rightarrow \C$. 
For this, observe that $\Re f$ and $\Im f$ are each $(r,\gamma)$-$L_2$-global. (See \cite[Definition 4.4]{KLM23} for the definition; this is the same concept that we have called $(r,\gamma)$-derivative-global in \cref{def:deriv-global} in the specific setting of functions on $B(G_{\mc{Q}})$.) Furthermore as 
\[|x+iy|^{p} = (x^2+y^2)^{p/2}\le 2^{p/2 - 1}(|x|^{p} + |y|^{p}),\]
we have that 
\begin{equation}\label{to-tensor} \snorm{T_{\rho} f}_p^{p}\le 2^{p/2 - 1} \Big(\snorm{T_{\rho} (\Re f)}_p^{p} + \snorm{T_{\rho} (\Im f)}_p^{p}\Big)\le 2^{p/2 - 1} \snorm{f}_2^2  \gamma^{p-2}.\end{equation}
Write $f^{\otimes m} : (\Omega^n)^m \rightarrow \C$ for the Cartesian product of $m$ copies of $f$. Note that $f^{\otimes m}$ is $(r,\gamma^{m})$-$L_2$-global, $\snorm{T_{\rho}f}_p^{m} =\snorm{T_{\rho}(f^{\otimes m})}_p$ and $\snorm{f}_2^m = \snorm{T_{\rho}(f^{\otimes m})}_2$.
Applying \cref{to-tensor} with $f^{\otimes m}$, taking $m$th roots, and sending $m\to\infty$ we obtain the desired inequality $\snorm{T_{\rho} f}_p^{p}\le \snorm{f}_2^2 \gamma^{p-2}$.

We next extend the scope of \cite[Corollary~4.7]{KLM23} from functions $f:\Omega^{n}\to \C$ to functions $f:\Omega_1\times \cdots \times \Omega_n\to \C$ on general product spaces. In order to do so, given $f:\Omega_1\times \cdots \times \Omega_n\to \C$, we may define $\wt{f}:(\Omega_1\times \cdots \times \Omega_n)^{n}\to \C$ via 
\[ \wt{f}\Big( \big( (x_{ij})_{j \in [n]}\big)_{i \in [n]} \Big) = f \big( (x_i)_{i \in [n]} \big).\]

We see that $\wt{f}$ is $(r,\gamma)$-$L_2$-global, $\snorm{T_{\rho}f}_p = \snorm{T_{\rho}\wt{f}}_p$, and $\snorm{\wt{f}}_2 = \snorm{f}_2$. This immediately implies the desired result.

\bibliographystyle{amsplain0}
\bibliography{main.bib}

\end{document}